\documentclass[11pt]{article}
\usepackage{graphicx}
\usepackage{amsmath,amssymb,amsthm,amsfonts,mathrsfs}
\usepackage{hyperref}
\usepackage[numbers,sort&compress]{natbib}
\usepackage{amssymb}

\usepackage{txfonts}
\usepackage{mathrsfs}
\usepackage{amssymb}
\usepackage{amsmath}
\usepackage{amsfonts}
\usepackage{amsfonts,enumerate}
\usepackage{pifont}
\usepackage{enumerate}
\usepackage{graphics}
\usepackage{verbatim}
\usepackage{amsthm}
\usepackage{amsmath}
\allowdisplaybreaks

\newtheorem{thm}{Theorem}[section]

\newtheorem{cor}[thm]{Corollary}
\newtheorem{lem}[thm]{Lemma}

\theoremstyle{definition}

\theoremstyle{remark}
\newtheorem{rem}{Remark}[section]

\numberwithin{equation}{section}

\DeclareMathSymbol{\C}{\mathalpha}{AMSb}{"43}

\newcommand{\N}{{\mathbb N}}
\newcommand{\R}{{\mathbb{R}}}

\newcommand{\bsub}{\begin{subequations}}
\newcommand{\esub}{\end{subequations}$\!$}

\begin{document}

\title{Existence of solution for Schr\"odinger-Newton system with vanishing magnetic potential
%\thanks {This work is supported by NSFC (No.12471107).}
}

\author{Liying Shan \thanks{School of Mathematics and Statistics, Hubei Key Laboratory of Mathematics \& Intelligent Computation and Interdisciplinary Sciences, Wuhan University of Technology, Wuhan 430070, P. R. China. Email: \texttt{shanly@whut.edu.cn}}
\qquad Wei Shuai \thanks{School of Mathematics and Statistics \& Key Laboratory of Nonlinear Analysis and Applications (Ministry of Education), Central China Normal University, Wuhan 430079, P. R. China. Email: \texttt{wshuai@ccnu.edu.cn}}
\qquad Ping Yang \thanks{School of Mathematics and Statistics, Central China Normal University, Wuhan 430079, P. R. China. Email: \texttt{yangping0427@163.com}}
\qquad Jianghua Ye \thanks{School of Mathematics \& Center for Applied Mathematics of Guangxi, Guangxi University, Nanning 530004, P. R. China. Email: \texttt{jhye@gxu.edu.cn}}
}

%\date{\date}

\smallbreak \maketitle

\begin{abstract}
In this paper, we study the following Schr\"odinger-Newton system with a vanishing magnetic potential and an electric potential
\begin{equation}\label{0.1}
\left\{
 \renewcommand{\arraystretch}{1.25}
 \begin{array}{ll}
 (-i\nabla +A(x))^2u+V(x)u=\Phi u, \ & x\in\R^3,\\
-\Delta \Phi =|u|^2, \ & x\in\R^3,
 \end{array}
 \right.
\end{equation}
where $i$ is the imaginary unit, $A:\R^3\to \R^3$ is a continuous magnetic potential, and $V:\R^3\to \R$ is a continuous electric potential. 
First, we show that the ground state energy of system \eqref{0.1} with $A(x)\equiv0$ and $V(x)\equiv V_{\infty}>0$ is isolated. Using this fact together with a global compactness lemma, a barycenter function, and Brouwer degree theory, we prove that under suitable conditions on $A$ and $V$, system \eqref{0.1} admits either a ground state solution or a high energy solution.
Interestingly, whether a ground state solution exists is significantly affected by the magnetic potential $A$ and the electric potential $V$. Furthermore, we provide sufficient conditions that guarantee the existence of a ground state solution, as well as conditions under which no ground state solution exists and a high energy solution appears instead.
%we give sufficient conditions under which system \eqref{0.1} admits a ground state solution, as well as sufficient conditions under which system \eqref{0.1} admits no ground state solution but possesses a high energy solution.
%When the magnetic potential vanishe and 
%When $A$ and $V$ satisfy certain conditions, we obtain that system \eqref{0.1} either has a ground state solution or a high energy solution by using variational methods.
%Furthermore,
%whether a ground state solution exists is significantly affected by the relationship between the potentials 
%and their limits.
\end{abstract}

\vskip 0.05truein


\vskip 0.05truein

%\vskip 0.08truein
%
%\noindent {\bf Data Availability Statement:}
%The data that supports the findings of this study are available within the article.

\vskip 0.05truein

\section{Introduction}
The present paper is concerned with the following nonlinear Schr\"odinger-Newton system with a magnetic potential and an electric potential
\begin{equation}\label{eqs1.1}
 \left\{
 \renewcommand{\arraystretch}{1.25}
 \begin{array}{ll}
 (-i\nabla +A(x))^2u+V(x)u=\Phi u, \ & x\in\R^3,\\
-\Delta \Phi =|u|^2, \ & x\in\R^3,
 \end{array}
 \right.
\end{equation}
where $i$ is the imaginary unit, $A:\R^3\to \R^3$ is a continuous magnetic potential, and $V:\R^3\to \R$ is a continuous electric potential. System \eqref{eqs1.1} originates from the study of physical phenomena in quantum mechanical systems that are subject to both nonlocal interactions and external electromagnetic fields \cite{FG,GHW}. For more physical interpretations and applications of the magnetic Schr\"odinger-Newton system, we refer the reader to \cite{Lions,J} and the references therein.

%Under the influence of the magnetic field $B=\nabla \times A$, %using the principle of minimal coupling, 
%the ordinary gradient operator $\nabla$ is replaced by the covariant derivative $\nabla_A:=-i\nabla +A(x)$, and the magnetic Schrödinger operator is defined by  
%\begin{align*}
%     -\Delta_A=(-i\nabla +A(x))^2=-\Delta-2iA(x)\cdot \nabla -i\operatorname{div} A(x)+|A(x)|^2.
    %(-i\nabla +A(x))^2=-\Delta-2iA(x)\cdot \nabla -i\operatorname{div} A(x)+|A(x)|^2.
%\end{align*}
%More discussion of the Schrödinger Hamiltonian $ (-i\nabla +A(x))^2$ may be found in \cite{A,La,Mi}.
% Furthermore, equation \eqref{eqs1.1} has a profound physical background and is used to describe the dynamics of an electrically charged particle in a magnetic field. 
% Specifically, it was first used to describe the quantum mechanics of a polaron at rest by Pekar \cite{P} in 1954. 
% In 1976, Choquard used it to describe an electron trapped in its own hole\cite{L}. 
% This equation also arises in many interesting situations related to the quantum theory of large systems of nonrelativistic bosonic atoms and molecules, see \cite{LS}. It was also used by Penrose \cite{Pen} to describe a nonlinear system obtained by coupling the linear Schr\"odinger equation of quantum mechanics with the gravitational law of Newtonian mechanics.  For more details, see \cite{A,Mi,Me,La} and their references. 

%and other background, see \cite{A,Mi,Me,La} and their references.

It is well-known that, for $u\in L^{\frac{12}{5}}(\mathbb{R}^3,\mathbb{C})$, by the Lax-Milgram theorem, the second equation of system \eqref{eqs1.1} has a unique solution $\Phi_{|u|}\in D^{1,2}(\mathbb{R}^3,\R)$, which can be written as 
\begin{align*}%\label{eq1.2}
        \Phi_{|u|}(x)=\frac{1}{4\pi}\int_{\mathbb{R}^3}\frac{|u(y)|^{2} }{|x-y|}dy.
\end{align*}
Then, system \eqref{eqs1.1} can be transformed into the following single equation with a nonlocal term
\begin{equation}\label{eqs1.1.0}
 (-i\nabla +A(x))^2u+V(x)u=\Phi_{|u|} u, \  x\in\R^3.
\end{equation}

If $A(x)\equiv0$, equation \eqref{eqs1.1.0}  reduces to 
\begin{equation}\label{1.2}
       -\Delta u+V(x)u=\Phi_{|u|} u, \  x\in\R^3.
\end{equation}
Equation \eqref{1.2} has a profound physical background. It was used to describe the quantum mechanics of a polaron at rest \cite{P} and the quantum theory of large systems of nonrelativistic bosonic atoms and molecules \cite{LS}. For more background, see \cite{Pen} and the references therein.
% Equation \eqref{1.2} was first used to describe the quantum mechanics of a polaron at rest by Pekar \cite{P} in 1954. In 1976, Choquard used the equation \eqref{1.2} to describe an electron trapped in its own hole\cite{L}. This equation also arises in many interesting situations related to the quantum theory of large systems of nonrelativistic bosonic atoms and molecules, see \cite{LS}. It was also used by Penrose \cite{Pen} to describe a nonlinear system obtained by coupling the linear Schr\"odinger equation of quantum mechanics with the gravitation law of Newtonian mechanics.
 Equation \eqref{1.2} has been widely investigated under various assumptions on the electric potential $V$.   %and its solution is usually considered to be real-valued.
% It is well-known that system \eqref{1.2} can be transferred to a single equation with a nonlocal term. In fact, for each $u\in L^{\frac{12}{5}}(\R^3,\mathbb{C})$, according to the Lax-Milgram theorem, there exists a unique $\Phi_{|u|}\in D^{1,2}(\mathbb{R} ^3)$ such that $-\Delta \Phi_{|u|}=|u|^{2}$, and
% \begin{align}\label{1.4}
%         \Phi_{|u|}(x)=\frac{1}{4\pi}\int_{\mathbb{R}^3}\frac{|u(y)|^{2} }{|x-y|}dy,
% \end{align}
% where $H^1(\R^3)$ is the usual Sobolev space equipped with the inner product and norm
% \[
% (u,v)_{H^1}=\ds\int_{\R^3}(\nabla u \nabla v+uv)dx,\ \ \|u\|=(u,u)_{H^1}^{\frac{1}{2}}.
% \]
 % substituting \eqref{1.4} into equation \eqref{1.2}, one obtains
 % \begin{align}\label{1.5}
 %         -\Delta u+u=\Phi_{|u|} u,\ \  u\in H^1(\R^3,\R).
 % \end{align}
% In the past several decades, equation \eqref{1.2} has been extensively investigated by many researchers.
For example, in the special case $V(x)\equiv1$, equation \eqref{1.2} becomes 
 \begin{equation}\label{1.5}
        -\Delta u+u=\Phi_{|u|} u, \  x\in\R^3,
 \end{equation}
 and the real-valued solutions of equation \eqref{1.5} have received particularly extensive attention in the literature.
 By using variational methods, Lieb \cite{L}, Lions \cite{Lions}, and Tod and Moroz \cite{TM} obtained the existence and uniqueness of the ground state solution for equation \eqref{1.5}.
%The existence and uniqueness of ground state solutions for equation \eqref{1.5} were obtained by Lieb \cite{L}, Lions \cite{Lions} and Tod-Moroz \cite{TM} using variational methods.
In \cite{MZ}, Ma and Zhao proved that every positive solution of equation \eqref{1.5} is radially symmetric and monotone decreasing about some fixed point by employing moving plane methods. Combining this radial symmetry  with the results of Lieb \cite{L}, Ma and Zhao subsequently established the uniqueness of the positive solution for equation \eqref{1.5}.
We refer the reader to \cite{DHL,LPW,SS,GLWY,WW} and the references therein for further results on equation \eqref{1.2}.

Obviously, equation \eqref{1.5} is a special case of the following Choquard equation
\begin{equation}\label{1.6}
        -\Delta u+u=(I_{\alpha}*|u|^p)|u|^{p-2} u, \ x\in\R^N,
\end{equation}
where $N\ge 1$, $\alpha\in (0,N)$, $p\in(1,+\infty)$ %$p\in(\frac{N+\alpha}{N},\frac{N+\alpha}{N-2})$ $\frac{1}{p}\in(\frac{N-2}{N+\alpha},\frac{N}{N+\alpha})$ 
and $I_{\alpha}$ is the Riesz potential defined by
\begin{align*}%\label{i}
        I_{\alpha}(x)=\frac{\varGamma(\frac{N-\alpha}{2})}{\varGamma(\frac{\alpha}{2}) \pi^{\frac{N}{2}} 2^{\alpha} |x|^{N-\alpha}},
\end{align*}
and $\varGamma$ is the Gamma function, see \cite{R}.
In \cite{MV}, Moroz and Van Schaftingen demonstrated the
existence of a ground state solution for equation \eqref{1.6} with 
  $\frac{1}{p}\in(0,1)\cap (\frac{N-2}{N+\alpha},\frac{N}{N+\alpha})$. %$p\in(\frac{N+\alpha}{N},\frac{N+\alpha}{N-2})$.
%They also %established both the regularity and positivity of this ground state solution, 
% provd that all positive ground state solutions are radially symmetric and exhibit monotone decay around a certain point. Furthermore, they derived the decay asymptotics of the ground state solution as $|x|$ approaches infinity.
Under the conditions that $N=3$, $\alpha=2$, $p>2$, and $p$ is sufficiently close to 2, Xiang \cite{X} established the uniqueness and nondegeneracy of the ground state solution for equation \eqref{1.6}.
For additional results on the Choquard equation, we refer the reader to \cite{MV1,MV3,CL,GHPS} and the references therein.

%Furthermore, if the potential is not constant and satisfies certain conditions, there are also many results.
%In \cite{CL}, Cao and Li studied the Choquard equation with the potential on the right for $\alpha=p=2$ and $N=3,4,5$. Under certain assumptions, they obtained that the ground state of the Choquard equation does not exist. However, by global compactness analysis, they proved that there exists a positive high energy solution, and the energy of the positive high energy solution is between one and two times that of the ground state solutions of the limit equation.
 
On the other hand, the Schr\"odinger-Newton system with magnetic was investigated in \cite{CC2,CSS,LiB2,LL1}. In \cite{CSS}, Cingolani et al. studied the semi-classical regime for 
 \begin{equation}\label{1.7}
      (-\varepsilon i\nabla -A(x))^2u+V(x)u=\frac{1}{\varepsilon ^2}\Phi_{|u|} u, \  x\in\R^3,
 \end{equation}
 under the assumptions that the magnetic potential $A$ is of class $C^1$ and the electric potential $V$ is a continuous function with several local minima. They proved that for any sufficiently small $\varepsilon>0$, equation \eqref{1.7} possesses a solution having multiple concentration regions which are located around the minimum points of the electric potential $V$.
 In \cite{CC2}, the authors considered equation \eqref{1.7} with $A$ and $V$ compatible with the action of a group $G$ of linear isometries of $\R^3$. They showed that, for
any given homomorphism $\tau: G\to \mathbb{S}^1$ (the unit complex numbers), there is a combined effect of the symmetries and the potential $V$ on the number of semiclassical solutions $u$ satisfying $u(gx)=\tau(g)u(x)$ for all $g\in G$, $x\in\R^3$.
% If $A(x) \nequiv 0$, there are also many results on Choquard equations. 
 %Moreover, %for Choquard equation, 
 %For $p=2$ and $N=3$, 
 Li, Lu and Luo \cite{LL1} investigated normalized solutions of the magnetic Schr\"odinger-Newton system, namely,
 \begin{equation}\label{e1}
      \left\{
 \renewcommand{\arraystretch}{1.25}
 \begin{array}{ll}
  - (\nabla -iA(x))^2u+V(x)u=\mu u+\Phi_{|u|} u, \ x\in\R^3,\\
     \int_{\R^3}|u|^2dx=m>0,
 \end{array}
 \right.
 \end{equation}
 where $V\in L^{\infty}(\R^3)$ and $A(x)=\frac{b}{2}(-x_2,x_1,0)$ with $b\in \R$. Under certain assumptions on $V$, they proved the existence of ground state solutions to equation \eqref{e1} by employing the concentration compactness principle. Furthermore, they provided a refined description of the asymptotic behaviour of ground states as $m\to +\infty$. %Additionally, they concluded that there exists a unique ground state, (up to a constant phase), which must be real-valued and cylindrically symmetric when $V(x)$ is cylindrically symmetric and $m$ is sufficiently large. 
%Salazar \cite{Sa} researched the following nonlinear equation
 % \begin{equation}\label{e3}
%  -\Delta_A u+(V(x)+V_{\infty})u=\left(\frac{1}{|x|^{\alpha}}*|u|^p\right)|u|^{p-2} u, \ x\in\R^N,
% \end{equation}
 %where $N\ge 3$, $\alpha\in (0,N)$, $p\in [2,\frac{2N-\alpha}{N-2})$, $A:\R^N\to \R^N$ is a vector potential, $V:\R^N\to \R$ is a scalar potential. 
 %For a given group $\Gamma$ of linear isometries of $\R^N$, Salazar assumed that $A(gx)= gA(x)$  and $V(gx) = V(x)$ for all $g \in \Gamma,\  x \in \R^N$. Under some assumptions on the decay of $A(x)$ and $V(x)$ at infinity, he established the existence of vortex-type solution $u$ to equation\eqref{e3} which satisfies
%$u(\gamma x) = \phi (\gamma)u(x)$ for all $\gamma \in \Gamma,\  x \in \R^N$,
%where $\phi:\Gamma\to S_1$ is a given continuous group homomorphism into the unit complex numbers.
 % Bueno et al.
 %Mamani, and Pereira 
% \cite{Bu} researched the following nonlinear Choquard equation
% \begin{equation}\label{e2}
%  -\Delta_A u+V(x)u=\left(\frac{1}{|x|^{\alpha}}*F(|u|)\right)\frac{f(|u|)}{|u|} u, \ x\in\R^N,
% \end{equation}
%where
%$\alpha\in(0,N)$, $A:\R^N\to \R^N$ is a vector potential, $V:\R^N\to \R$ is a scalar potential, 
%$f:\R\to\R $ and $F$ is the primitive of $f$. 
% Under mild hypotheses, they proved the existence of a ground state solution for equation \eqref{e2}. They also proved a simple multiplicity result by applying Ljusternik–Schnirelmann methods.
 For results about the Choquard equation with magnetic potential, readers can see \cite{CC,J,B1,CW,CS} and the references therein.

Recently, Li and Wei  \cite{LiB} studied the following Schr\"odinger  equation
 \begin{align}
     \label{01}
     (-i\nabla +A(x))^2u+V(x)u=|u|^{q-2}u, \  x\in\R^N,
 \end{align}
 where $N\ge 3$, $q\in(2,\frac{2N}{N-2})$, $A$ and $V$ satisfy the following conditions: 
 
 $(A)$\ \ $A\in C(\R^N,\R^N)$, $\lim \limits_{|x|\to +\infty} A(x)=0$;
 
 $(V)$\ \ $V\in C(\R^N,\R)$, $V_0:=\inf \limits_{x\in \R^N} V(x)>0$, $\lim \limits_{|x|\to +\infty} V(x)=V_{\infty}$;\\
 and some additional conditions.
  % $(H)$\ \   $V_{\infty}<|A(x)|^2+V(x)< V_{\infty}+c_0 $, where $c_0$ is a constant that is related to the ground state energy of the limiting equation.
 % there exist $c_0>0$ and $k_0>2\sqrt{V_{\infty}}$ such that 
 % $$
 % |A(x)|^2+V(x)\le V_{\infty}+c_0e^{-k_0|x|}\  \text{for} \ x \in\R^N. 
 % $$
Then, by using the variational methods, they obtained that the equation \eqref{01} either has a ground state solution or has a high energy solution.  
% The result differs from previous ones, which were determined to be either ground state solutions or high energy solutions, and the energy of the high energy solution was between one and two times that of the ground state solution of the limiting equation.
As far as we know, there seems to be no relevant result for the Schr\"odinger-Newton system with magnetic potential.

Inspired by the above research, in this paper, we study the existence of nontrivial solutions to equation \eqref{eqs1.1.0} under the condition that the potentials $A$ and $V$ satisfy $(A)$ and $(V)$ with $N=3$. 
It is worth noting that for the real-valued case (see \cite{BL,C,CV,GHPS}), the uniqueness of the positive solution to the limit equation is crucial for establishing the existence of high energy solutions. However, for the complex-valued case, in order to obtain high energy solutions, we need the isolation of the ground state energy of the limit equation.
%\vskip 0.08 truein
 % Moreover, it is worth noting that for the real-valued case, likely following \cite{CL}, the uniqueness of the positive solution of the limit equation is crucial for establishing the existence of high energy solutions. For the complex-valued case, in order to obtain high energy solutions, we need the isolation of the ground state energy of the limit equation. 
 The limit equation of equation \eqref{eqs1.1.0} is as follows
%in order to obtain the existence of nontrivial solutions of equation \eqref{eqs1.1.0}, we need to discuss the properties of the following limit equation
\begin{equation}\label{1.8.2}
        \left\{
 \renewcommand{\arraystretch}{1.25}
 \begin{array}{ll}
 -\Delta u+V_{\infty}u=\Phi_{|u|} u, \ x\in \R^3,\\
u\in H^1(\mathbb{R}^3,\mathbb{C}),
 \end{array}
 \right.
\end{equation}
where
\begin{align*}
    H^1(\R^3,\mathbb{C}):=\left\{u\in L^2(\R^3,\mathbb{C})\ |\ \nabla u\in L^2(\R^3,\mathbb{C}^3)\right\}
  \end{align*}  
  is a Hilbert space equipped with the inner product
  \begin{align*}
    (u,\varphi ):=\mathrm{Re}\left(\int_{\R^3}\left(\nabla u \overline{\nabla \varphi}+V_{\infty}u\overline{\varphi}\right)dx\right),
\end{align*}
and the norm $\|u\|=(u,u)^{\frac{1}{2}}$.
 
 However, to the best of our knowledge, there is no result concerning the isolation of the ground state energy of the equation \eqref{1.8.2}.
Therefore, the first objective of this paper is to fill this gap, namely, to prove the isolation of the ground state energy of the equation  \eqref{1.8.2}. 
For this purpose, define the energy functional of equation \eqref{1.8.2} by
\begin{align*}
        I_{\infty}(u):=\int_{\R^3}\left[\frac{1}{2}\big(|\nabla u|^2+V_\infty|u|^2 \big)-\frac{1}{4}\Phi_{|u|} |u|^2\right]dx,\ \ u\in H^1(\mathbb{R}^3,\mathbb{C}).
\end{align*}
It is obvious that $I_{\infty}$ is well-defined in $H^1(\mathbb{R}^3,\mathbb{C})$ and belongs to $C^1(H^1(\mathbb{R}^3,\mathbb{C}),\mathbb{R})$. Moreover, the Gateaux derivative of $I_{\infty}$ is given by
\begin{align*}
  \langle I'_{\infty}(u),\varphi\rangle=\mathrm{Re}\left(\int_{\R^3}\left[\nabla u \overline{\nabla \varphi}+V_{\infty}u\overline{\varphi}-\Phi_{|u|}u\bar{\varphi}\right]dx\right).
\end{align*}
Denote the ground state energy of equation \eqref{1.8.2} by
\begin{equation}\label{nc1}
        c_{\infty}:= \inf \big\{I_{\infty}(u)\ |\ u\in \mathcal{N}_{\infty} \big\},
\end{equation}
where 
\begin{align}\label{nc2}
        \mathcal{N}_{\infty}:= \big\{u\in H^1(\mathbb{R}^3,\mathbb{C})\backslash \{0\} \ |\ J_{\infty}(u):=\langle I'_{\infty}(u),u\rangle=0\big\}.
\end{align}

% \begin{equation}\label{1.8.1}
%         \left\{
%  \renewcommand{\arraystretch}{1.25}
%  \begin{array}{ll}
%  -\Delta u+V_{\infty}u=\Phi_u u, \ \text{in}\ \R^3,\\
% u\in H^1(\mathbb{R}^3).
%  \end{array}
%  \right.
% \end{equation}
% By Theorem A, the positive solution of equation \eqref{1.8.1} is uniquely determined up to translations. Denote %$u_{\infty}(x):=V_{\infty}\bar{u}_{\infty}(\sqrt{V_{\infty}}x)$,
% \begin{align}\label{1.8.3}
%     u_{\infty}(x):=V_{\infty}\bar{u}_{\infty}(\sqrt{V_{\infty}}x),
% \end{align}
% then $u_{\infty}$ is a positive ground state solution of equation \eqref{1.8.1}.

%\vskip 0.08truein
The first result of this paper is as follows.
%The following lemma shows that $c_{\infty}$ is a isolated critical value of $I_{\infty}$, which is the key to obtain a high energy solution of equation \eqref{eqs1.1.0} .   

\begin{thm}\label{th1.0}
      There exists $\tau >0$ such that if $u$ is a nontrivial solution of equation \eqref{1.8.2} with $I_{\infty}(u)\leq c_{\infty}+\tau$, then there exists $\theta \in \mathbb{R}$ and $x_0 \in \mathbb{R}^3$ such that $u=e^{i\theta }u_{\infty}(\cdot -x_0)$,
where 
\begin{align}\label{1.8.3}
    u_{\infty}(x):=V_{\infty}\bar{u}_{\infty}(\sqrt{V_{\infty}}x)%,\ u_{\infty} \in H^1(\R^3,\R).
\end{align}
and $\bar{u}_{\infty}$ the unique positive solution of equation \eqref{1.5} with $\bar{u}_{\infty}(0)=\max_{x\in \R^3}\bar{u}_{\infty}(x)$.
% and then $u_{\infty}$ is a positive ground state solution of the follwing equation
% \begin{equation}\label{1.8.1}
%         \left\{
%  \renewcommand{\arraystretch}{1.25}
%  \begin{array}{ll}
%  -\Delta u+V_{\infty}u=\Phi_u u, \ \text{in}\ \R^3,\\
% u\in H^1(\mathbb{R}^3).
%  \end{array}
%  \right.
% \end{equation}
\end{thm}

The proof of Theorem \ref{th1.0} is inspired by \cite{B}. 
We first show that the ground state solution of \eqref{1.8.2} is unique %determined by $u_{\infty}$ 
up to translations and complex phases. %by showing $c_{\infty}=\tilde{c}_{\infty}$ 
%the ground energy $c_{\infty}$ is equivalent to the real-valued ones  $\tilde{c}_{\infty}$ 
%(see Lemma \ref{lem:1}).
Then to prove Theorem \ref{th1.0}, we assume by contradiction that $c_{\infty}$ is not an isolated critical level of $I_{\infty}$, that is, there exists a critical value sequence $\{c_n\}$ of  $I_{\infty}$ satisfying $c_n\searrow c_{\infty}$ and a corresponding critical point sequence $\{u_n\}$. %$\{u_n\}\subset H^1(\R^3,\mathbb{C})$. 
Next, we derive that $\{u_n\}$ is compact in $H^1(\R^3,\C)$ up to translations and complex phases. %that is, there exist $\{x_n\}\subset\R^3$ and $\theta_0\in\R$ such that $\tilde{u}_n:=e^{-i\theta_0}u_n(\cdot+x_n)\to u_{\infty}$ in $H^1(\R^3,\C)$. %(see Lemma \ref{lem:3}). 
%$\tilde{u}_n\to u_{\infty}$ strongly in $H_A^1(\R^3,\C)$, where $\tilde{u}_n$ is obtained from $u_n$ by translations and complex phases 
Furthermore, using Brouwer degree theory, we succeed in showing that there exists $\{(\tilde{x}_n,t_n)\}\subset \R^3\times \R$ %with $(\tilde{x}_n,t_n)\rightarrow\vec{0}$ 
such that $\hat{u}_n:=e^{it_n}u_n(\cdot+\tilde{x}_n) \perp \mathrm{Ker}L$, %$\hat{u}_n:=e^{it_n}\tilde{u}_n(\cdot+\tilde{x}_n) \perp \mathrm{Ker}L$, 
where $L$ is the linearized operator of equation \eqref{1.8.2} at $u_{\infty}$ (see \eqref{L}). 
Finally, define $w_n:=\frac{\hat{u}_n-u_{\infty}}{\left\|\hat{u}_n-u_{\infty}\right\|}$; then $\|w_n\|=1$ and $w_n\perp \mathrm{Ker}L$. 
%and $w_n$ is a bounded sequence in $H^1(\R^3,\C)$. 
On the one hand, there exists $w_0\in H^1(\R^3,\C)$ such that $w_n \rightharpoonup w_0$ in $H^1(\R^3,\C)$, which gives $w_0\perp \mathrm{Ker}L$ and $w_0\in \mathrm{Ker}L$. Hence, %and moreover  
$w_0=0$. 
On the other hand, we verify $\|w_0\|=\lim\limits_{n\rightarrow\infty}\|w_n\|=1$, which implies a contradiction and proves Theorem \ref{th1.0}. 
However, due to the presence of the nonlocal term, some delicate computation is needed (see Proof of Theorem \ref{th1.0}).
%$\hat{u}_n \perp \mathrm{Ker}L$ and $u_{\infty}\perp \mathrm{Ker}L$, 
%we know $w_0 \perp \mathrm{Ker}L$. Hence, $w_0=0$, 
%Moreover, we also prove $ w_n \to w_0$ in $H^1(\R^3,\C)$, then $\|w_0\|=1$. This is a contradiction. 

With Theorem \ref{th1.0} in hand, we are able to investigate the existence of nontrivial solutions to equation \eqref{eqs1.1.0}.
%Before stating our main results, we give some notations.
Define the work space
\begin{align*}
    H^1_A(\R^3,\mathbb{C}):=\left\{u\in L^2(\R^3,\mathbb{C})\ |\ \nabla_A u\in L^2(\R^3,\mathbb{C}^3)\right\},
\end{align*}
which is a Hilbert space endowed with the inner product
\begin{align*}
    (u,\varphi )_A:=\mathrm{Re}\left(\int_{\R^3}\left(\nabla_A u \overline{\nabla_A \varphi}+V(x)u\overline{\varphi}\right)dx\right),
\end{align*}
and the norm $\|u\|_A:=(u,u)_A^{\frac{1}{2}}$, where $\nabla_A u:=-i\nabla u+A(x)u$.
%\begin{align*}
%\|u\|_A:=\left(\int_{\R^3}(|\nabla_A u|^2 +V(x)|u|^2)dx\right)^{\frac{1}{2}}.
%\end{align*}
%Introduce the diamagnetic inequality \cite[theorem 7.21]{LL}, %for each $u\in H_A^1(\R^3,\C)$, 
%In addition, for $u\in L^p(\R^3,\mathbb{C})$, denote the norm 
%\begin{align*}
%     \|u\|_{L^p(\R^3)}:=\left(\int_{\R^3}|u|^pdx\right)^{\frac{1}{p}}.
%\end{align*}
%For simplicity, we note
%$\|u\|:=\|u\|_{H^1(\R^3,\mathbb{C})}$ and  $\|u\|_{L^p(\R^3)}:=\|u\|_{L^p(\R^3,\mathbb{C})}$.
Define the energy functional corresponding to equation \eqref{eqs1.1.0} as follows
\begin{align*}
       I(u):=\int_{\R^3}\left[\frac{1}{2}\big(|\nabla_A u|^2+V(x)|u|^2 \big)-\frac{1}{4}\Phi_{|u|} |u|^2\right]dx,\ \ u\in H^1_A(\R^3,\mathbb{C}).
\end{align*}
By the diamagnetic inequality \cite[Theorem 7.21]{LL}
\begin{align}\label{nor}
    |\nabla |u||\le |\nabla_A u|,\ u\in H_A^1(\R^3,\C),
\end{align}
we have $|u|\in H^1(\R^3,\R)$ for $u\in H_A^1(\R^3,\C)$, which, together with the Sobolev embedding theorem, implies that the embedding $H_A^1(\R^3,\mathbb{C})\hookrightarrow L^p(\R^3,\mathbb{C})$ is continuous for $p\in[2,6]$.
According to the Hardy-Littlewood-Sobolev inequality (see Lemma \ref{lm2.0.2}),  $I$ is well-defined in $H_A^1(\R^3,\mathbb{C})$ and $I\in C^1(H_A^1(\R^3,\mathbb{C}),\R)$.
Denote the ground state energy of equation \eqref{eqs1.1.0} by
\begin{equation}\label{1.11}
        c:= \inf \big\{I(u)\ |\ u\in \mathcal{N} \big\},
\end{equation}
where
\begin{align}\label{1.12}
    \mathcal{N}:= \big\{u\in H^1_A(\R^3,\mathbb{C})\backslash \{0\} \ |\ J(u):=\langle I'(u),u\rangle=0\big\}.
\end{align}
We know that %(see in Lemma \ref{lm3.1})%According to Lemma \ref{lm3.1}, 
if $u\in \mathcal{N}$ is a minimizer of $c$, then $u$ is a ground state solution of equation \eqref{eqs1.1.0} (see Lemma \ref{lm3.1}).

%On this basis, our results regarding equation \eqref{eqs1.1.0} are as follows.
%main results are as follows.
Our first result regarding equation \eqref{eqs1.1.0} is as follows.

\begin{thm}\label{th1.4}
       Suppose $(A)$ and $(V)$ hold with $N=3$. If $A$ and $V$ satisfy 
       
       %$(H_1)$\ \ $|A(x)|^2+V(x)\leq V_{\infty}+k_1e^{-k_2|x|}$,
       %$(H_1)$\ \ $|A(x)|^2+V(x)\leq \left(\frac{c_{\infty}+\min\{c_{\infty},\tau\}}{c_{\infty}}\right)^{\frac{1}{2}}V_{\infty}$,
       $(H_1)$\ \ $|A(x)|^2+V(x)\leq V_{\infty}+4\|u_{\infty}\|_{L^2(\R^3)}^{-2}\left(\sqrt{c_{\infty}(c_{\infty}+\min\{c_{\infty},\tau\})}-c_{\infty}\right)$,
       
      \noindent{where}
      %\begin{align}\label{c}
      %    c=4\|u_{\infty}\|_{L^2(\R^3)}^{-2}\left(\sqrt{c_{\infty}(c_{\infty}+\min\{c_{\infty},\tau\})}-c_{\infty}\right),
     % \end{align}
      $\tau $ and $u_{\infty}$ are given in Theorem \ref{th1.0}, then equation \eqref{eqs1.1.0}  either has a ground state solution, or has a high energy solution $u$ with $I(u)\in (c_{\infty},c_{\infty}+\min\{c_{\infty},\tau\})$.
\end{thm}

One of the difficulties in proving Theorem \ref{th1.4} is the lack of compactness of the embedding $H_A^1(\R^3,\mathbb{C})\hookrightarrow L^p(\R^3,\mathbb{C})$. First, we show $c\le c_{\infty}$. Then, in order to recover the compactness, we establish a global compactness lemma (see Lemma \ref{lm4.2}), which provides a complete description of Palais-Smale sequences for the energy functional $I$. Note that the Br\'ezis–Lieb type splitting properties of the functional $\int_{\R^3}\Phi_{|u|} |u|^2dx$ and its derivative play an important role in the proof of the global compactness lemma. However, to the best of our knowledge, there is no result on the Br\'ezis–Lieb type splitting property of the derivative of $\int_{\R^3}\Phi_{|u|} |u|^2dx$ in $H^1(\R^3,\mathbb{C})$. We fill this gap (see \eqref{A2}). 
Next, if $c<c_{\infty}$, using the global compactness lemma, we obtain the compactness of the minimizing sequence for $c$, proving that equation \eqref{eqs1.1.0} has a ground state solution. 

On the other hand, if equation \eqref{eqs1.1.0} has no ground state solution, we will show that it possesses a high energy solution. To this end, by employing the global compactness lemma and Theorem \ref{th1.0}, we succeed in recovering the compactness of $(PS)_{d}$ sequence for $I$ with $d\in (c_{\infty},c_{\infty}+\min\{c_{\infty},\tau\})$. The second challenge lies in finding a $(PS)_{d}$ sequence for $I$ in this range. Inspired by \cite{BL,C}, utilizing a barycenter function and Brouwer degree theory, we construct a special minimax value $\bar{c}\in (c_{\infty},c_{\infty}+\min\{c_{\infty},\tau\})$ (see \eqref{3.29}). Then, the general minimax principle yields a $(PS)_{\bar{c}}$ sequence. Hence, equation \eqref{eqs1.1.0} admits a high energy solution $u$ with $I(u)=\bar{c}$.

Theorem \ref{th1.4} provides an alternative conclusion: equation \eqref{eqs1.1.0} has either a ground state solution or a high energy solution. 
We are thus led to the following question: under what conditions does equation \eqref{eqs1.1.0} have a ground state solution?
The following result answers this question.
%concerns the ground state solution of equation \eqref{eqs1.1.0}.
\begin{thm}\label{th1.1}
        Suppose $(A)$ and $(V)$ hold with $N=3$. If $A$ and $V$ satisfy 
        
       $(H_2)$\ \ $|A(x)|^2+V(x)\leq V_{\infty}$ and $|A(x)|^2+V(x)\nequiv V_{\infty}$,
       
        \noindent{then} equation \eqref{eqs1.1.0} has a ground state solution.
\end{thm}

\begin{rem}
In \cite{LiB}, Li and Wei proved that if $A$ and $V$ satisfy $(A)$, $(V)$ and 

$(H')$\ \ there exist $c_0>0$ and $k_0>0$ such that $|A(x)|^2+V(x)\le V_{\infty}-c_0 e^{-k_0|x|}$ for $x\in \R^N$,
   % $$|A(x)|^2+V(x)\le V_{\infty}-c_0 e^{-k_0|x|}\ \ \text{for}\ x\in \R^N, $$
   
    \noindent{then} equation \eqref{01} admits a ground state solution. Notice that $(H_2)$ is weaker than $(H')$. Following a similar argument as in Theorem \ref{th1.1}, we can improve their result. Namely, if $A$ and $V$ satisfy $(A)$, $(V)$ and $(H_2)$, then equation \eqref{01} possesses a ground state solution.
   % Note that if $(A)$ and $(V)$ hold, $A$ and $V$ satisfy the following condition\\
  %  Li and Wei \cite{LiB} also proved that equation \eqref{01} admits a ground state solution. Obviously, $(H_2)$ is weaker than $(H')$, we promote their result.
\end{rem}

%By Theorem \ref{th1.1}, we know that equation \eqref{eqs1.1.0} has a ground state solution if $(A)$, $(V)$ and $(H_2)$ hold.  
%$\rho(x)$ satisfies $(\rho 1)$.
Next, a natural question is under which conditions equation \eqref{eqs1.1.0} does not admit a ground state solution but instead possesses a high energy solution. This aspect is addressed in our next results. %The following result concerns the ground state solution of equation  \eqref{eqs1.1.0}.
%A natural question is whether equation \eqref{eqs1.1.0} still has a ground state solution if $\rho\leq 0$. 
%The following result gives a negative answer.
\begin{thm}
    \label{th1.2}
        Suppose $(A)$ and $(V)$ hold with $N=3$. If $V$ satisfies
        
        $(H_3)$\ \ $ V_{\infty}< V(x)$ \ for all $x\in \R^3$,% and $V(x)\nequiv V_{\infty}$,
       %$(H_3)$\ \ $|A(x)|^2+V(x)> V_{\infty}$, % and $|A(x)|^2+V(x)\nequiv V_{\infty}$,
       
        \noindent{then} equation \eqref{eqs1.1.0} has no ground state solution.
\end{thm}
% \begin{rem}
%     By $(H_3)$, we know that the following condition holds
    
%     $(H_4)$\ \ $V_{\infty}< |A(x)|^2+V(x)$  for all $x\in \R^3$.
    
%     \noindent{So,} $(H_3)$ is stronger. We point out that if the ground state solution of equation \eqref{eqs1.1.0} is uniquely determined by the real-valued one up to translations and complex phases, $(H_3)$ can be replaced by $(H_4)$.
% \end{rem}

%The relationship between the ground state energy of equation \eqref{eqs1.1.0} and that of equation \eqref{1.8.2} plays an important role in the proof of Theorems \ref{th1.1} and \ref{th1.2}. 
%More precisely, if $(H_2)$ holds, by proving $c<c_\infty$, equation \eqref{eqs1.1.0} has a ground state solution.
% While, if $(H_3)$ holds, we find the fact $c=c_\infty$, which is vital to establishing the nonexistence of ground state solutions for equation \eqref{eqs1.1.0}. Furthermore, we can prove the existence of high energy solutions. 
Combining Theorems \ref{th1.4} and \ref{th1.2}, we obtain the following result.
\begin{thm}\label{th1.3}
       Suppose $(A)$, $(V)$, $(H_1)$ and $(H_3)$ hold with $N=3$. Then equation \eqref{eqs1.1.0} has a high energy solution $u$ with $I(u)\in (c_{\infty},c_{\infty}+\min\{c_{\infty},\tau\})$.
\end{thm}

\vskip 0.1truein
% \begin{rem}
%          It is pointed out that if $V(x)\equiv V_{\infty}$, Theorem \ref{th1.4} holds.
%          Moreover, for $A\equiv0$, Theorems \ref{th1.4}, \ref{th1.1}, \ref{th1.2} and \ref{th1.3} also hold.
% \end{rem}

\begin{rem}
        Under some obvious modifications, Theorems \ref{th1.1} and \ref{th1.2} hold for the following Choquard equation with a vanishing magnetic potential
        \begin{equation}\label{1.8}
                 (-i\nabla +A(x))^2u+V(x)u=\big[I_{\alpha}*|u|^p\big]|u|^{p-2} u, \ x\in\R^N,
        \end{equation}
        where $N\ge 3$, $\alpha\in (0,N)$ and $p\in (\frac{N+\alpha}{N},\frac{N+\alpha}{N-2})$. Moreover, if the positive ground state solution of equation \eqref{1.6} is unique up to translations and is nondegenerate, then Theorems \ref{th1.0}, \ref{th1.4}
         and \ref{th1.3} also hold for equation \eqref{1.8}.
\end{rem}

This paper is organized as follows. In Section \ref{s2.1}, we prove Theorem \ref{th1.0}, showing that $c_{\infty}$ is an isolated critical value of $I_{\infty}$.
Section \ref{s2} is devoted to the proof of Theorem \ref{th1.1}. Specifically, if $c<c_{\infty}$, we prove that equation \eqref{eqs1.1.0} has a ground state solution by using a global compactness lemma.
If, instead, equation \eqref{eqs1.1.0} admits no ground state solution, utilizing a barycenter function, Brouwer degree theory, the global compactness lemma, and Theorem \ref{th1.0}, we show the existence of a high energy solution for the equation \eqref{eqs1.1.0}.
We give the proofs of Theorems \ref{th1.1} and \ref{th1.2} in Section \ref{s4}.
%We first show that $c\le c_{\infty}$ holds. %with $(A)$ and $(V)$. 
%For $c<c_{\infty}$,  Next, we can prove $c=c_{\infty}$ if equation \eqref{eqs1.1.0} has no ground state solution. 
%Then, we construct a special min-max value $\bar{c}\in (c_{\infty},c_{\infty}+\min\{c_{\infty},\tau\})$. 
%Finally, combining the general minimax principle, the global compactness lemma, and Theorem \ref{th1.0}, we prove the existence of a high energy solution for the equation \eqref{eqs1.1.0}.
 
%has been studied extensively by many researchers, This causes some mathematical difficulties, which make the study of (\ref{eqs1.1}) particularly interesting.

%equation \eqref{eqs1.1.0} appears in the study of standing waves of the form $\varPsi(t,x)=e^{-iu_{\infty} t}u(x)$ to the following
%time-dependent nonlinear Schr\"odinger-Newton system
%\begin{equation}\label{eqs1.0.0}
%        \left\{\begin{array}{ll}
%          -i\varPsi_t-\Delta \varPsi+(1+u_{\infty})\varPsi=\Phi \varPsi, \hspace{3mm}(t,x)\in \R_+\times \R^3,\\
%          -\Delta \Phi  =| \varPsi| ^2+\rho(x), \hspace{3mm}(t,x)\in \R_+\times \R^3,
%        \end{array} \right.
%\end{equation}
%where $u_{\infty}\in\R$ is the frequency constant.

\section{Proof of Theorem \ref{th1.0}}\label{s2.1}
In this section, we prove Theorem 1.1. Firstly,  we introduce the Hardy-Littlewood-Sobolev inequality (see \cite[Thorem $A^*$]{SW}  and \cite[(1.1)]{L1}), which will be used later.%\cite[Lemma A.1]{ZZ} that \cite[Theorem 4.3]{LL}
%we present a property of the function $\Phi_{|u|}$ defined by \eqref{eq1.2}.
%\begin{lem}\label{lm2.0.3}
%There exists $C>0$ such that for all $u,v\in  L^{\frac{12}{5}}(\mathbb{R}^3,\mathbb{C})$, it holds 
%\begin{align*}
%    \left\|\Phi_{|u|}-\Phi_{|v|}\right\|_{D^{1,2}(\R^3)}\leq C\left\||u|^2-|v|^2\right\|_{L^{\frac{6}{5}}(\mathbb{R}^3)}.
%\end{align*}
%\end{lem}
%\noindent{\bf Proof.}
%We give the proof for completeness. Notice that $\Phi_{|u|}-\Phi_{|v|}$ satisfies 
%\begin{align*}
%    -\Delta (\Phi_{|u|}-\Phi_{|v|})=|u|^2-|v|^2,
%\end{align*}
%which, together with the H\"older inequality and the Sobolev inequality, implies 
%\begin{align*}
%    \left\|\Phi_{|u|}-\Phi_{|v|}\right\|_{D^{1,2}(\R^3)}^2
%    =&\int_{\R^3}\left(|u|^2-|v|^2\right)(\Phi_{|u|}-\Phi_{|v|})dx
%    \leq \left\||u|^2-|v|^2\right\|_{L^{\frac{6}{5}}(\mathbb{R}^3)} \left\|\Phi_{|u|}-\Phi_{|v|}\right\|_{L^{6}(\R^3)}\\
%    \leq &C\left\||u|^2-|v|^2\right\|_{L^{\frac{6}{5}}(\mathbb{R}^3)} \left\|\Phi_{|u|}-\Phi_{|v|}\right\|_{D^{1,2}(\R^3)}.
%\end{align*}
%This completes the proof.
%\qed

\begin{lem}\label{lm2.0.2}
    If $u\in L^p(\R^3,\R)$, $1<p<q<+\infty$, and $\frac{1}{q}=\frac{1}{p}+\frac{\lambda}{3}-1$ with $\lambda\in(0,3)$, then there exists $C_{p,q,\lambda}>0$, independent of $u$, %, but may depend on $p$, $q$ and $\lambda$, 
    such that
    \begin{align}\label{HLS1}
        \left\|\int_{\R^3}\frac{u(y)}{|x-y|^{\lambda}}dy\right\|_{L^q(\R^3)}\le C_{p,q,\lambda}\|u\|_{L^p(\R^3)}.
    \end{align}
    Moreover, combining the H\"older inequality, for all $v\in L^s(\R^3,\R)$ with  $1<s<+\infty$ and $\frac{1}{s}+\frac{1}{p}+\frac{\lambda}{3}=2$, it holds
    \begin{align}\label{HLS2}
         \left|\int_{\R^3}\int_{\R^3}\frac{u(y)v(x)}{|x-y|^{\lambda}}dydx\right|
    \le C \|v\|_{L^s(\R^3)}\left\|\int_{\R^3}\frac{u(y)}{|x-y|^{\lambda}}dy\right\|_{L^q(\R^3)}\le C \|v\|_{L^s(\R^3)}\|u\|_{L^p(\R^3)}.
    \end{align}
\end{lem}

Now, we discuss the 
relationship between the ground state energy of equation \eqref{1.8.2} and that of the follwing equation
\begin{equation}\label{1.8.1}
        \left\{
 \renewcommand{\arraystretch}{1.25}
 \begin{array}{ll}
 -\Delta u+V_{\infty}u=\Phi_{|u|} u, \ \text{in}\ \R^3,\\
u\in H^1(\mathbb{R}^3,\R).
 \end{array}
 \right.
\end{equation}
% By Theorem A, the positive solution of equation \eqref{1.8.1} is uniquely determined up to translations.
% Denote the energy functional of \eqref{1.8.2} by
% \begin{align}
%         I_{\infty}(u):=\int_{\R^3}\left[\frac{1}{2}\big(|\nabla u|^2+V_\infty|u|^2 \big)-\frac{1}{4}\Phi_u |u|^2\right]dx,\ \ u\in H^1(\mathbb{R}^3,\mathbb{C}).
% \end{align}
% It is obvious that $I_{\infty}(u)$ is well-defined and belongs to $C^1(H^1(\mathbb{R}^3,\mathbb{C}),\mathbb{R})$, then the Gateaux derivative is given by
% \begin{align}
%   \langle I'_{\infty}(u),\varphi\rangle=\mathrm{Re}\left(\int_{\R^3}\left[\nabla u \overline{\nabla \varphi}+V_{\infty}u\overline{\varphi}-\Phi_uu\bar{\varphi}\right]dx\right).
% \end{align}
% We define the ground state energy of equation \eqref{1.8.2} by
% \begin{equation*}
%         c_{\infty}:= \inf \big\{I_{\infty}(u)\ |\ u\in \mathcal{N}_{\infty} \big\},
% \end{equation*}
% where 
% \begin{align*}
%         \mathcal{N}_{\infty}:= \big\{u\in H^1(\mathbb{R}^3,\mathbb{C})\backslash \{0\} \ |\ J_{\infty}(u):=\langle I'_{\infty}(u),u\rangle=0\big\}.
% \end{align*}
Define the Nehari manifold of equation \eqref{1.8.1} and the corresponding ground state energy by
\begin{align*}
\widetilde{\mathcal{N}}_{\infty} :=\left\{u\in H^1(\R^3,\R)\setminus\{0\}: J_{\infty}(u)=0\right\},\ \ \tilde{c}_{\infty}:=\inf\limits_{u\in \widetilde{\mathcal{N}}_{\infty}}I_{\infty}(u),
\end{align*}
and summarize the results in \cite{L,Lions,TM,MZ} as follows.

\begin{lem}\label{thm A} 
     The positive solution of equation \eqref{1.8.1} is uniquely determined up to translations. Moreover, the positive solution of equation \eqref{1.8.1} is also a ground state solution of equation \eqref{1.8.1} and must be radially symmetric and monotone decreasing about some fixed point.  
\end{lem}

Therefore, the function $u_{\infty}$ defined by \eqref{1.8.3} is the unique positive solution of equation \eqref{1.8.1} with $u_{\infty}(0)=\max_{x\in \R^3}u_{\infty}(x)$, which is also a ground state solution. 
Note that $u_{\infty}\in \widetilde{\mathcal{N}}_{\infty}\subset \mathcal{N}_{\infty}$, where $\mathcal{N}_{\infty}$ is defined by \eqref{nc2}. So, $c_{\infty}$ defined by \eqref{nc1} is well-defined.
%Then,  combining Lemma \ref{thm A}, we know that the Nehari manifold $\mathcal{N}_{\infty}$ is non-empty.
Before proving $c_{\infty}=\tilde{c}_{\infty}$, we show some properties of the Nehari manifold $\mathcal{N}_{\infty}$.

\begin{lem}\label{lm2.0.1}
    For $v\in H^1(\mathbb{R}^3,\mathbb{C})\setminus \{0\}$, there exists a unique $t_v>0$ such that $t_vv\in \mathcal{N}_{\infty}$. Moreover, $t_v\in (0,1)$ if $J_{\infty}(v)<0$, while $t_v>1$ if $J_{\infty}(v)>0$. %$I_{\infty}(tv)<I_{\infty}(t_vv)$ for $t>0$ and $t\neq t_v$.
     %Especially, $\mathcal{N}$ is non-empty.
\end{lem}
\noindent{\bf Proof.}
For each $v\in H^1(\mathbb{R}^3,\mathbb{C})\setminus \{0\}$ and $t\in \R^+$, define
\begin{align*}
        g_v(t):=I_{\infty}(tv)=\frac{t^2}{2}\int_{\R^3}(|\nabla v|^2+V_{\infty}|v|^2)dx-\frac{t^4}{4}\int_{\R^3}\Phi_{|v|} |v|^2 dx.
\end{align*}
One has
\begin{align*}
        g'_v(t)=t\left(\int_{\R^3}(|\nabla v|^2+V_{\infty}|v|^2)dx-t^2\int_{\R^3}\Phi_{|v|} |v|^2 dx\right),
\end{align*}
and then $g_v$ has a unique critical point% $t_v$.
\begin{align*}%\label{N1}
    t_v=\|v\|\left(\int_{\R^3}\Phi_{|v|} |v|^2 dx\right)^{-\frac{1}{2}}.
\end{align*}
Note that $J_{\infty}(t_v v)=t_vg'_v (t_v)=0$. So, $t_vv\in \mathcal{N}_{\infty}$. Moreover, $t_v\in (0,1)$ if $J_{\infty}(v)<0$, while $t_v>1$ if $J_{\infty}(v)>0$.
\qed

\vskip 0.08truein
Then, we show that the ground state energy of equation \eqref{1.8.1} is equal to that of equation \eqref{1.8.2}.

\begin{lem}\label{lem:1}
%We have the following identity:
%$$c_{\infty}=\tilde{c}_{\infty},$$
%where $c_{\infty}$ is defined by \eqref{nc1}. 
It holds $c_{\infty}=\tilde{c}_{\infty}$. Moreover, if $u$ is a ground state solution of equation \eqref{1.8.2}, then there exists $\theta_0\in\R$ and $x_0\in\R^3$ such that $u(x)=e^{i\theta_0}u_{\infty}(x-x_0)$. %where $u_{\infty}$ is defined by \eqref{1.8.3}.
%and then $u_{\infty}$ is a ground state solution of equation \eqref{1.8.1}.
\end{lem}

\begin{proof}
Since $\widetilde{\mathcal{N}}_{\infty}\subset\mathcal{N}_{\infty}$, we have $c_{\infty}\leq \tilde{c}_{\infty}$ .  Note that $\nabla|u|=\mathrm{Re}\left(\frac{\bar{u}}{|u|}\nabla u\right)$, then for $i=1,2,3$, 
\begin{align}\label{eq:01}
\left|\frac{\partial|u|}{\partial x_i}\right|=\left|\mathrm{Re}\left(\frac{\bar{u}}{|u|}\frac{\partial u}{\partial x_i}\right)\right|\leq\left|\frac{\bar{u}}{|u|}\frac{\partial u}{\partial x_i}\right|=\left|\frac{\partial u}{\partial x_i}\right|,
\end{align}
that is, 
\begin{equation}\label{eq:1}
    \left|\nabla |u|\right|\leq\left|\nabla u\right|.
\end{equation}
Suppose that $u\in H^1(\R^3,\mathbb{C})$ is a ground state solution of equation \eqref{1.8.2}, then 
\begin{align*}
    J_{\infty}(|u|)=\int_{\R^3}\left(\left|\nabla|u|\right|^2+V_{\infty}|u|^2-\Phi_{|u|}|u|^2\right)dx\leq\int_{\R^3}\left(|\nabla u|^2+V_{\infty}|u|^2-\Phi_{|u|}|u|^2\right)dx=0.
\end{align*}
So, by Lemma \ref{lm2.0.1}, we know that there exists $t\in(0,1]$ such that $t|u|\in\widetilde{\mathcal{N}}_{\infty}$. Then, one has
\begin{align*}
    \tilde{c}_{\infty}&\leq I_{\infty}(t|u|)=I_{\infty}(t|u|)-\frac{1}{4}J_{\infty}(t|u|)=\frac{t^2}{4}\int_{\R^3}\left(\left|\nabla|u|\right|^2+V_{\infty}|u|^2\right)dx
    \leq\frac{t^2}{4}\int_{\R^3}\left(\left|\nabla u\right|^2+V_{\infty}|u|^2\right)dx\\
    &\leq\frac{1}{4}\int_{\R^3}\left(\left|\nabla u\right|^2+V_{\infty}|u|^2\right)dx=c_{\infty}
    \leq \tilde{c}_{\infty},
\end{align*}
which implies that $c_{\infty}=\tilde{c}_{\infty}=I_{\infty}(|u|)$, $t=1$, and 
\begin{align}\label{eq:2}
\int_{\R^3}\left|\nabla|u|\right|^2dx=\int_{\R^3}\left|\nabla u\right|^2dx.
\end{align}
Therefore, $|u|$ is a nonnegative ground state solution of equation \eqref{1.8.1}. Then, using Theorem 3 in \cite{MV}, we obtain $|u|>0$. %$|u|\in C^2(\R^3,\R)$. By the strong maximum principle \cite[Chapter 6,Theorem 4]{E}, we obtain $|u|>0$.

Moreover, it follows from Lemma \ref{thm A} that $|u|(x)=u_{\infty}(x-x_0)$ for some $x_0\in\R^3$. For $z\in\mathbb{C}$, denote 
\begin{equation*}
    \mathrm{sign}(z):=
    \left\{
    \begin{aligned}
&\frac{z}{|z|}, &z\neq0;\\
&0, &z=0.
    \end{aligned}
    \right.
\end{equation*}
Note that $u=\mathrm{sign}(u)|u|=\mathrm{sign}(u) u_{\infty}(x-x_0)$. Since $|u|>0$, then $\mathrm{sign}(u)=\frac{u}{|u|}$. It suffices to prove $\nabla \mathrm{sign}(u)=0$ and then we can get $\mathrm{sign}(u)=e^{i\theta}$. In fact,
\begin{align*}
    \nabla\mathrm{sign}(u)=\nabla\left(\frac{u}{|u|}\right)=\frac{\nabla u|u|-u\nabla|u|}{|u|^2}.
\end{align*}
Combining \eqref{eq:1} and \eqref{eq:2}, we have $\left|\nabla|u|\right|=|\nabla u|$, which together with \eqref{eq:01}, gives 
$$\frac{\partial|u|}{\partial x_i}=\frac{\bar{u}}{|u|}\frac{\partial u}{\partial x_i}\in \R,\ \ i=1,2,3.$$ 
Thus,
$$u\frac{\partial |u|}{\partial x_i}=\frac{u\bar{u}}{|u|}\frac{\partial u}{\partial x_i}=|u|\frac{\partial u}{\partial x_i}, \ \ i=1,2,3.$$
It follows that $u\nabla|u|=|u|\nabla u$, so $\nabla\mathrm{sign}(u)=0$. The proof is completed.
\end{proof}

%\vskip 0.08truein

The next lemma shows that the sequence of critical points has a strongly convergent subsequence in $H^1(\R^3,\mathbb{C})$ up to translations and complex phases.
\begin{lem}\label{lem:3}
    Suppose $\{c_n\}\subset[c_{\infty},+\infty)$ is a critical value sequence of $I_{\infty}$ satisfying $c_n\searrow c_{\infty}$ and let $\{u_n\}\subset H^1(\R^3,\mathbb{C})$ be the corresponding critical point sequence. Then, there exist $\theta_0\in\R$ and $\{x_n\}\subset\R^3$ such that up to a subsequence,
    \begin{align*}
        e^{-i\theta_0}u_n(x+x_n)\rightarrow u_{\infty}\ \ \text{strongly in} \ \ H^1(\R^3,\mathbb{C}).
    \end{align*}
\end{lem}
\begin{proof}
For $u\in\mathcal{N}_{\infty}$, by the Hardy-Littlewood-Sobolev inequality \eqref{HLS2} and the Sobolev inequality, we have
    \begin{align*}
       %\min\{1,V_{\infty}\}
       \|u\|^2=
       \int_{\R^3} \left(|\nabla u|^2+V_{\infty}|u|^2\right)dx=\int_{\R^3}\Phi_{|u|}|u|^2dx\leq C\|u\|^4,
    \end{align*}
which implies that there exists $\alpha>0$ such that 
\begin{align}\label{eq:3}
    \|u\|\geq\alpha,\ \ \text{for}\ \ u\in\mathcal{N}_{\infty}.
\end{align}
%\begin{align}\label{eq:3}
 %   \|u\|_{H^1(\R^3)}\geq\alpha,\ \ \text{for}\ \ u\in\mathcal{N}_{\infty}.
%\end{align}
Moreover,
\begin{align*}
    c_{\infty}+o_n(1)=c_n=I_{\infty}(u_n)=I_{\infty}(u_n)-\frac{1}{4}J_{\infty}(u_n)=\frac{1}{4}\|u_n\|^2.
    %\int_{\R^3} \left(|\nabla u_n|^2+V_{\infty}|u_n|^2\right)dx%\ge \frac{\min\{1,V_{\infty}\}}{4}\|u_n\|^2.
\end{align*}
Thus, $\{u_n\}$ is bounded in $H^1(\R^3,\mathbb{C})$. Denote
\begin{align*}
\delta_0:=\lim\limits_{n\rightarrow\infty}\sup\limits_{z\in\R^3}\int_{B(z,1)}|u_n(x)|^2dx,
\end{align*}
where $B(z,1)$ is an open ball with $z$ as the center and 1 as the radius. (Hereafter, we denote $B(0,r)$ by $B_r$ for simplicity.) %If the center is 0, we also simply write it as $B_1$.) 
If $\delta_0=0$, by the vanishing lemma \cite[Lemma 1.21]{W}, $|u_n|\rightarrow0$ in $L^p(\R^3,\R)$, $p\in(2,6)$. Then,
\begin{align*}
     %\min\{1,V_{\infty}\}\|u_n\|^2\leq\int_{\R^3} \left(|\nabla u_n|^2+V_{\infty}|u_n|^2\right)dx
     \|u_n\|^2=\int_{\R^3}\Phi_{|u_n|}|u_n|^2dx\leq C\|u_n\|^4_{L^{\frac{12}{5}}(\R^3)}=o_n(1),
\end{align*}
which contradicts \eqref{eq:3}. This implies that $\delta_0>0$. So there exists $\{\bar{x}_n\}\subset\R^3$ such that 
\begin{align*}
    \int_{B(\bar{x}_n,1)}|u_n|^2dx\geq\frac{\delta_0}{2}.
\end{align*}
Define $\bar{u}_n(x):=u_n(x+\bar{x}_n)$. It is obvious that $\|\bar{u}_n\|=\|u_n\|\leq C$, which implies that there exists $\bar{u}\in H^1(\R^3,\mathbb{C})$ such that, up to a subsequence,
\begin{equation*}
\left\{
    \begin{aligned}
        &\bar{u}_n\rightharpoonup\bar{u}\ \ \text{in}\ \ H^1(\R^3,\mathbb{C}),\\
        &\bar{u}_n\rightarrow\bar{u}\ \ \text{in}\ \ L^p_{loc}(\R^3,\mathbb{C}),p\in[2,6),\\
        &\bar{u}_n\rightarrow\bar{u}\ \ \text{a.e. on}\ \ \R^3.
    \end{aligned}
    \right.
\end{equation*}
Then
\begin{align*}
\int_{B(0,1)}|\bar{u}|^2dx=\lim\limits_{n\rightarrow\infty}\int_{B(0,1)}|\bar{u}_n|^2dx=\lim\limits_{n\rightarrow\infty}\int_{B(\bar{x}_n,1)}|u_n|^2dx\geq\frac{\delta_0}{2}.
\end{align*}
Thus, we conclude that $\bar{u}\neq0$. Since %$I_{\infty}(\bar{u}_n)=I_{\infty}(u_n)$ and 
$\langle I'_{\infty}(\bar{u}_n),\varphi\rangle=\langle I'_{\infty}(u_n),\varphi(x-\bar{x}_n)\rangle=0$ for each $\varphi\in H^1(\R^3,\mathbb{C})$, $\bar{u}_n$ is also a critical point of $I_{\infty}$, which implies that $\bar{u}$ is a nontrivial solution of equation \eqref{1.8.2}. By the weakly lower semi-continuity of norm, we can get
\begin{align*}
    c_{\infty}\leq I_{\infty}(\bar{u})=
    \frac{1}{4}\|\bar{u}\|^2%\int_{\R^3} \left(|\nabla \bar{u}|^2+V_{\infty}|\bar{u}|^2\right)dx
    \leq \frac{1}{4}\lim\limits_{n\rightarrow\infty}\|\bar{u}_n\|^2%\int_{\R^3} \left(|\nabla \bar{u}_n|^2+V_{\infty}|\bar{u}_n|^2\right)dx\\
    =\frac{1}{4}\lim\limits_{n\rightarrow\infty}\|u_n\|^2%\int_{\R^3} \left(|\nabla u_n|^2+V_{\infty}|u_n|^2\right)dx
    =\lim\limits_{n\rightarrow\infty}I_{\infty}(u_n)=\lim\limits_{n\rightarrow\infty}c_n=c_{\infty},
\end{align*}
which implies that $I_{\infty}(\bar{u})=c_{\infty}$ and $\bar{u}_n\rightarrow\bar{u}$ in $H^1(\R^3,\mathbb{C})$. Thus, $\bar{u}$ is a ground state solution of equation \eqref{1.8.2}. By Lemma \ref{lem:1}, $\bar{u}(x)=e^{i\theta_0}u_{\infty}(x-\bar{x}_0)$ for some $\theta_0\in \R$ and $\bar{x}_0\in\R^3$. Define $x_n=\bar{x}_n+\bar{x}_0$ and 
\begin{align}\label{eq:4}
\tilde{u}_n(x)=e^{-i\theta_0}u_n(x+x_n),
\end{align}
then
\begin{align*}
    \|\tilde{u}_n-u_{\infty}\|=\|e^{-i\theta_0}u_n(x+\bar{x}_n+\bar{x}_0)-u_{\infty}\|=\|\bar{u}_n-\bar{u}\|\rightarrow0,\ \ \text{as} \ \ n\rightarrow\infty,
\end{align*}
that is, $e^{-i\theta_0}u_n(x+x_n)\rightarrow u_{\infty}$ in $H^1(\R^3,\mathbb{C})$.   
\end{proof}

% For convenience, we denote the inner product and the norm by
% \begin{align*}
%     (u,v)_{\infty}:=\mathrm{Re}\left(\int_{\R^3}\left(\nabla u\overline{\nabla v}+V_{\infty}u\bar{v}\right)dx\right),\ \ \|u\|_{\infty}:=(u,u)^{\frac{1}{2}}_{\infty}.
% \end{align*}
% It is easy to see that the norm $\|u\|_{\infty}$ is equivalent to the standard norm $\|u\|_{H^1(\R^3)}$.
Note that the linearized operator of equation \eqref{1.8.2} at $u_{\infty}$ is given by
\begin{align}\label{L}
    Lw:=-\Delta w+V_{\infty}w-\Phi_{|u_{\infty}|}w-\frac{1}{2\pi}u_{\infty}\int_{\R^3}\frac{\mathrm{Re}\left(u_{\infty}(y)\bar{w}(y)\right)}{|x-y|}dy
\end{align}
acting on $L^2(\R^3,\mathbb{C})$ with domain $H^1(\R^3,\mathbb{C})$. The nondegeneracy of ground state solution for equation \eqref{1.8.2} has been studied by Lenzmann \cite{Lenzmann2009}, as shown below.
\begin{lem}\cite[Remark (1)]{Lenzmann2009}\label{lem:2}
    The linearized operator $L$ satisfies
    \begin{align*}
        \mathrm{Ker} L=\mathrm{span}\left\{\frac{\partial u_{\infty}}{\partial x_1},\frac{\partial u_{\infty}}{\partial x_2},\frac{\partial u_{\infty}}{\partial x_3},iu_{\infty}\right\}.
    \end{align*}
\end{lem}

Next, we modify the phase and translation of $\tilde{u}_n$ slightly to obtain the following orthogonality relations.
\begin{lem}\label{lem:4}
Suppose that $\{c_n\}$, $\{u_n\}$ satisfy the assumption in Lemma \ref{lem:3} and $\tilde{u}_n$ is defined in \eqref{eq:4}. Then there exists $\{t_n\}\subset\R$ and $\{\tilde{x}_n\}\subset\R^3$ such that 
\begin{align*}
    &e^{it_n}\tilde{u}_n(x+\tilde{x}_n)\rightarrow u_{\infty}\ \ \text{strongly in}\ \ H^1(\R^3,\mathbb{C}),\\
    &\left(e^{it_n}\tilde{u}_n(x+\tilde{x}_n),\frac{\partial u_{\infty}}{\partial x_k}\right)=0,\ \ k=1,2,3,\\
    &\left(e^{it_n}\tilde{u}_n(x+\tilde{x}_n),iu_{\infty}\right)=0.
\end{align*}
\end{lem}
\begin{proof}
Define the auxiliary map $\psi_n:\R^3\times \R\rightarrow\R^4$ by
\begin{align*}
&\psi_n(z,t):=\left(\psi^1_n(z,t),\psi^2_n(z,t),\psi^3_n(z,t),\psi^4_n(z,t)\right),\\
%\end{align*}
%\begin{align*}
&\psi^k_n(z,t):=\left(e^{it}\tilde{u}_n(x+z),\frac{\partial u_{\infty}}{\partial x_k}\right),\ \ k=1,2,3,\ \ \ \psi^4_n(z,t):=\left(e^{it}\tilde{u}_n(x+z),iu_{\infty}\right).
\end{align*}
By Lemma \ref{lem:3}, there exists $\psi(z,t)=(\psi^1_\ast,\psi^2_\ast,\psi^3_\ast,\psi^4_\ast)$ such that 
\begin{align*}
\psi_n(z,t)\rightarrow \psi(z,t)\ \ \text{uniformly on}\ \ \R^4,
\end{align*}
where
\begin{align*}
   \psi^k_\ast(z,t)=\left(e^{it}u_{\infty}(x+z),\frac{\partial u_{\infty}}{\partial x_k}\right),\ \  k=1,2,3,\ \ \ \psi^4_\ast(z,t)=\left(e^{it}u_{\infty
    }(x+z),iu_{\infty}\right).
\end{align*}
In fact, for $k=1,2,3$,
\begin{align*}
    \left|\psi^k_n(z,t)-\psi^k_\ast(z,t)\right|\leq\left\|e^{it}\tilde{u}_n(x+z)-e^{it}u_{\infty}(x+z)\right\|\left\|\frac{\partial u_{\infty}}{\partial x_k}\right\|=\left\|\tilde{u}_n-u_{\infty}\right\|\left\|\frac{\partial u_{\infty}}{\partial x_k}\right\|\rightarrow0,
\end{align*}
uniformly for $(z,t)\in \R^4$, as $n\rightarrow\infty$. Obviously,
\begin{align*}
    \psi(0,0)=\left(\left(u_{\infty},\frac{\partial u_{\infty}}{\partial x_1}\right),\left(u_{\infty},\frac{\partial u_{\infty}}{\partial x_2}\right),\left(u_{\infty},\frac{\partial u_{\infty}}{\partial x_3}\right),\left(u_{\infty},iu_{\infty}\right)\right).
\end{align*}
For $k=1,2,3$,
\begin{align*}
    \left(u_{\infty},\frac{\partial u_{\infty}}{\partial x_k}\right)=\int_{\R^3}\left(\nabla u_{\infty}\nabla\frac{\partial u_{\infty}}{\partial x_k}+V_{\infty}u_{\infty}\frac{\partial u_{\infty}}{\partial x_k}\right)dx=\int_{\R^3}\Phi_{|u_{\infty}|}u_{\infty}\frac{\partial u_{\infty}}{\partial x_k}dx=0,
\end{align*}
since the integrand function is odd with respect to $x_k$. And
\begin{align*}
    \left(u_{\infty},iu_{\infty}\right)=\mathrm{Re}\left(\int_{\R^3}\left(\nabla u_{\infty}\nabla(iu_{\infty})+V_{\infty}u_{\infty}iu_{\infty}\right)dx\right)=0.
\end{align*}
Thus,
\begin{align*}
    \psi(0,0)=(0,0,0,0).
\end{align*}
For $k=1,2,3$,
\begin{align*}
    \psi_\ast^k(z,t)=\cos t\int_{\R^3}\left(\nabla u_{\infty}(x+z)\nabla\frac{\partial u_{\infty}}{\partial x_k}+V_{\infty}u_{\infty}(x+z)\frac{\partial u_{\infty}}{\partial x_k}\right)dx.
\end{align*}
A direct compupation implies that for $k,j=1,2,3$,
\begin{align*}
    \frac{\partial\psi_\ast^k}{\partial z_j}(0,0)&=\int_{\R^3}\left(\nabla \frac{\partial u_{\infty}}{\partial x_j}\nabla\frac{\partial u_{\infty}}{\partial x_k}+V_{\infty}\frac{\partial u_{\infty}}{\partial x_j}\frac{\partial u_{\infty}}{\partial x_k}\right)dx=
    \left\{
\begin{aligned}
    &\left\|\frac{\partial u_{\infty}}{\partial x_k}\right\|^2,\ \ &j=k;\\
    &0,\ \ &j\neq k,
\end{aligned}
    \right.
    \\
    \frac{\partial\psi_\ast^k}{\partial t}(0,0)&=0.
\end{align*}
For $\psi_\ast^4$, we have
\begin{align*}
    &\psi_\ast^4(z,t)=\sin t\int_{\R^3}\left(\nabla u_{\infty}(x+z)\nabla u_{\infty}+V_{\infty}u_{\infty}(x+z)u_{\infty}\right)dx,\\
    &\frac{\partial\psi_\ast^4}{\partial z_j}(0,0)=0,\ \ j=1,2,3,\\
    &\frac{\partial\psi_\ast^4}{\partial t}(0,0)=\int_{\R^3}\left(|\nabla u_{\infty}|^2+V_{\infty}u_{\infty}^2\right)dx=\left\|u_{\infty}\right\|^2.
\end{align*}
Therefore, we conclude that
\begin{align*}
    D\psi(0,0)=
    \begin{pmatrix}
        \left\|\frac{\partial u_{\infty}}{\partial x_1}\right\|^2 &  &  &  \\
          &  \left\|\frac{\partial u_{\infty}}{\partial x_2}\right\|^2 &  &  \\
          &  &  \left\|\frac{\partial u_{\infty}}{\partial x_3}\right\|^2 & \\
          &  &  & \left\|u_{\infty}\right\|^2
    \end{pmatrix}.
\end{align*}

%For every small $r>0$, we claim that $\psi(z,t)\neq\vec{0}$ on $\overline{B_r}\setminus\{0\}$. In fact, 
Moreover, since $u_{\infty}$ is radial, it holds that
\begin{align*}
    \left|D\psi(0,0)\right|=\left\|\frac{\partial u_{\infty}}{\partial x_1}\right\|^6 \left\|u_{\infty}\right\|^2.
\end{align*}
Combining the continuity of $D\psi(z,t)$, there exists $r>0$ such that
\begin{align*}
    \left|D\psi(z,t)\right|\geq\frac{1}{2}\left\|\frac{\partial u_{\infty}}{\partial x_1}\right\|^6 \left\|u_{\infty}\right\|^2,\ \ \text{for}\ \ (z,t)\in \overline{B_r}.
\end{align*}
Then, we claim that $\psi(z,t)\neq\vec{0}$ on $\overline{B_r}\setminus\{0\}$.
Suppose, to the contrary, that there exists $(z^{*},t^{*})\in \overline{B_r}\setminus\{0\}$ such that $\psi(z^{*},t^{*})=\vec{0}$. Then, by the differential mean value theorem, there exists $\theta\in(0,1)$ such that
\begin{align*}
    (z^{*},t^{*})D\psi(\theta z^{*},\theta t^{*})=\psi(z^{*},t^{*})-\psi(0,0)=\vec{0},
\end{align*}
which implies that $(z^{*},t^{*})=(0,0)$. This is a contradiction. Thus, the Brouwer degree $\deg (\psi,B_r,\vec{0})$ is well-defined and 
\begin{align*}
   \deg (\psi,B_r,\vec{0})=\mathrm{sgn}\left|D\psi(0,0)\right|=1.
\end{align*}
Define
$
\delta_1:=\min\limits_{(z,t)\in\partial B_r}\left|\psi(z,t)\right|>0,
$
%\begin{align*}
 %   \deg (\psi,B_r,0)=\mathrm{sgn}\left|D\psi(0,0)\right|=1,\ \ \ \delta_1:=\min\limits_{(z,t)\in\partial B_r}\left|\psi(z,t)\right|>0.
%\end{align*}
since $\psi_n\rightarrow \psi$ uniformly on $\R^4$, then for $n$ large enough,
\begin{align*}
    \min\limits_{(z,t)\in\partial B_r}\left|\psi_n(z,t)\right|\geq\frac{\delta_1}{2},
\end{align*}
which means that $\deg (\psi_n,B_r,\vec{0})$ is well-defined. Define the homotopy
\begin{align*}
    H_n:\overline{B_r}\times[0,1]\rightarrow\R^4,\ \ H_n(z,t,s):=s\psi(z,t)+(1-s)\psi_n(z,t).
\end{align*}
It is easy to check that for $s\in[0,1]$, $H_n(z,t,s)\neq\vec{0}$ on $\partial B_r$. By the homotopy invariance of the topological degree, we have
\begin{align*}
\deg(\psi_n,B_r,\vec{0})=\deg(H_n(z,t,0),B_r,\vec{0})=\deg(H_n(z,t,1),B_r,\vec{0})=\deg(\psi,B_r,\vec{0})=1.
\end{align*}
It yields that there exists $\left\{(\tilde{x}_n,t_n)\right\}\subset B_r$ such that $\psi_n(\tilde{x}_n,t_n)=\vec{0}$, that is,
\begin{align*}
    e^{it_n}\tilde{u}_n(x+\tilde{x}_n)\perp \mathrm{Ker}L.
\end{align*}

%Now we will prove $(\tilde{x}_n,t_n)\rightarrow0$ in order to get the strongly congervence. 
Choose $\{r_i\}\subset(0,r)$ with $r_i\rightarrow0$. For each $i$, there exists $\{(\tilde{x}_{n}^{(i)},t_{n}^{(i)})\}\subset B_{r_i} $ such that $\psi_{n}(\tilde{x}_{n}^{(i)},t_{n}^{(i)})=\vec{0}$ for $n$ large enough.
By a diagonal argument, we extract a subsequence, still denoted by $\{(\tilde{x}_n,t_n)\}$, which converges to $\vec{0}$.
%Using a standard diagonal argument to extract a subsequence,
% still denote $\{(\tilde{x}_{n_i},t_{n_i})\}$ by $\{(\tilde{x}_n,t_n)\}$, we can get $(\tilde{x}_n,t_n)\rightarrow\vec{0}$. 
Denote 
\begin{align*}%\label{u}
    \hat{u}_n(x):=e^{it_n}\tilde{u}_n(x+\tilde{x}_n),
\end{align*}
then, it holds $\hat{u}_n\perp \mathrm{Ker}L$ 
and
\begin{align*}
    \left\|\hat{u}_n-u_{\infty}\right\|&= \left\|e^{it_n}\tilde{u}_n(x+\tilde{x}_n)-u_{\infty}\right\|\\
    &\leq\left\|e^{it_n}\tilde{u}_n(x+\tilde{x}_n)-e^{it_n}u_{\infty}(x+\tilde{x}_n)\right\|+\left\|e^{it_n}u_{\infty}(x+\tilde{x}_n)-u_{\infty}\right\|\\
    &=\left\|\tilde{u}_n-u_{\infty}\right\|+o_n(1)=o_n(1).
\end{align*}  
\end{proof}

Lemma \ref{lem:4} gives a nice property for the sequence of critical points $\{u_n\}$, that is,
\begin{align*}
    \hat{u}_n(x):=e^{it_n}\tilde{u}_n(x+\tilde{x}_n)\perp \mathrm{Ker}L.
\end{align*}

\noindent\textbf{Proof of Theorem \ref{th1.0}.} 
Suppose by contradiction that $c_{\infty}$ is not an isolated critical level of $I_{\infty}$. Then there exists a critical value sequence $\{c_n\}$ satisfying $c_n\searrow c_{\infty}$ and a corresponding critical point sequence $\{u_n\}\subset H^1(\R^3,\mathbb{C})$. Let $\hat{u}_n$ be the function defined above. Denote 
\begin{align*}
    w_n:=\frac{\hat{u}_n-u_{\infty}}{\left\|\hat{u}_n-u_{\infty}\right\|}.
\end{align*}
Since $\hat{u}_n$ and $u_{\infty}$ are both solutions of equation \eqref{1.8.2}, we have
\begin{align*}
    -\Delta(\hat{u}_n-u_{\infty})+V_\infty(\hat{u}_n-u_{\infty})&=\Phi_{|\hat{u}_n|}\hat{u}_n-\Phi_{|u_{\infty}|}u_{\infty}\\
    &=\Phi_{|\hat{u}_n|}(\hat{u}_n-u_{\infty})+\frac{1}{4\pi}u_{\infty}\int_{\R^3}\frac{|\hat{u}_n(y)|^2-|u_{\infty}(y)|^2}{|x-y|}dy\\
    &=\Phi_{|\hat{u}_n|}(\hat{u}_n-u_{\infty})+\frac{1}{4\pi}u_{\infty}\int_{\R^3}\frac{\mathrm{Re}\left(\big[\hat{u}_n(y)+u_{\infty}(y)\big]\big[\overline{\hat{u}_n}(y)-u_{\infty}(y)\big]\right)}{|x-y|}dy.
\end{align*}
Dividing both sides by $\left\|\hat{u}_n-u_{\infty}\right\|$, we get
\begin{align}\label{eq:6}
    -\Delta w_n+V_\infty w_n=\Phi_{|\hat{u}_n|}w_n+\frac{1}{4\pi}u_{\infty}\int_{\R^3}\frac{\mathrm{Re}\left((\hat{u}_n(y)+u_{\infty}(y))\overline{w_n}(y)\right)}{|x-y|}dy.
\end{align}
Note that $\|w_n\|=1$. So, there exists $w_0\in H^1(\R^3,\mathbb{C})$ such that, up to a subsequence,
\begin{equation}\label{2.12}
\left\{
    \begin{aligned}
        &w_n\rightharpoonup w_0\ \ \text{in}\ \ H^1(\R^3,\mathbb{C}),\\
        &w_n\rightarrow w_0\ \ \text{in}\ \ L^p_{loc}(\R^3,\mathbb{C}),p\in[2,6),\\
        &w_n\rightarrow w_0\ \ \text{a.e. on}\ \ \R^3.
    \end{aligned}
    \right.
\end{equation}
%Since $u_{\infty}$ is real-valued, Then, $\mathrm{Re}(u_{\infty}\bar{w}_0)=\mathrm{Re}(u_{\infty}w_0)$,
%Moreover, for each $\varphi\in H^1(\R^3,\mathbb{C})$, it holds
%\[
%\lim_{n\to\infty}\mathrm{Re}\left(\int_{\R^3}\overline{w_n}\varphi dx\right)=\lim_{n\to\infty}\mathrm{Re}\left(\int_{\R^3}w_n\overline{\varphi}dx\right)=\mathrm{Re}\left(\int_{\R^3}w_0\overline{\varphi}dx\right)=\mathrm{Re}\left(\int_{\R^3}\overline{w_0}\varphi dx\right).
%\]
Since $\hat{u}_n\to u_{\infty}$ in $H^1(\R^3,\mathbb{C})$, we deduce from \eqref{eq:6} and \eqref{2.12} that
 $w_0$ satisfies
\begin{align}\label{eq:7}
    -\Delta w_0+V_\infty w_0=\Phi_{|u_{\infty}|}w_0+\frac{1}{4\pi}u_{\infty}\int_{\R^3}\frac{\mathrm{Re}\left(2u_{\infty}(y)\overline{w_0}(y)\right)}{|x-y|}dy,
\end{align}
which %together with $\mathrm{Re}(u_{\infty}\bar{w}_0)=\mathrm{Re}(u_{\infty}w_0)$, 
yields that
\begin{align}\label{eq:5}
    w_0\in \mathrm{Ker}L.
\end{align}
According to Lemma \ref{lem:4}, 
\begin{align*}
    \left(\hat{u}_n,\frac{\partial u_{\infty}}{\partial x_k}\right)=0,\ \ \ \left(\hat{u}_n,iu_{\infty}\right)=0,\ \ \ \left(u_{\infty},\frac{\partial u_{\infty}}{\partial x_k}\right)=0,\ \ \ \left(u_{\infty},iu_{\infty}\right)=0,\ k=1,2,3.
\end{align*}
So there holds
\begin{align*}
    \left(w_n,\frac{\partial u_{\infty}}{\partial x_k}\right)=0,\ \ \ \left(w_n,iu_{\infty}\right)=0,\ k=1,2,3.
\end{align*}
Thus,
\begin{align*}
    \left(w_0,\frac{\partial u_{\infty}}{\partial x_k}\right)=0,\ \ \ \left(w_0,iu_{\infty}\right)=0,\ k=1,2,3,
\end{align*}
which gives $w_0\perp \mathrm{Ker}L$.
Combining with \eqref{eq:5}, we conclude $w_0=0$.

Next, we prove $\|w_0\|=1$. By \eqref{eq:6} and \eqref{eq:7}, 
\begin{align*}
    \|w_n\|^2&=\int_{\R^3}\Phi_{|\hat{u}_n|}|w_n|^2dx+\frac{1}{4\pi}\int_{\R^3}\int_{\R^3}\frac{\mathrm{Re}\left[(\hat{u}_n(y)+u_{\infty}(y))\overline{w_n}(y)\right]}{|x-y|}dyu_{\infty}(x)\mathrm{Re}(w_n(x))dx,\\
    \|w_0\|^2&=\int_{\R^3}\Phi_{|u_{\infty}|}|w_0|^2dx+\frac{1}{4\pi}\int_{\R^3}\int_{\R^3}\frac{\mathrm{Re}\left(2u_{\infty}(y)\overline{w_0}(y)\right)}{|x-y|}dyu_{\infty}(x)\mathrm{Re}(w_0(x))dx.
\end{align*}
Taking $R$ large enough, by the H\"older inequality, the Sobolev inequality and the Hardy-Littlewood-Sobolev inequality \eqref{HLS1}, we have
\begin{align*}
    &\left|\int_{\R^3}\Phi_{|\hat{u}_n|}|w_n|^2dx-\int_{\R^3}\Phi_{|u_{\infty}|}|w_0|^2dx\right|\\
    \leq&\left|\int_{\R^3}\left(\Phi_{|\hat{u}_n|}-\Phi_{|u_{\infty}|}\right)|w_n|^2dx\right|+\left|\int_{\R^3}\Phi_{|u_{\infty}|}\left(|w_n|^2-|w_0|^2\right)dx\right|\\
    \leq&\left\|\Phi_{|\hat{u}_n|}-\Phi_{|u_{\infty}|}\right\|_{L^6(\R^3)}\left\|w_n\right\|^2_{L^{\frac{12}{5}}(\R^3)}
    +\left\|\Phi_{|u_{\infty}|}\right\|_{L^6(B_R)}\left\||w_n|^2-|w_0|^2\right\|_{L^{\frac{6}{5}}(B_R)}\\
   & +\left\|\Phi_{|u_{\infty}|}\right\|_{L^6(\R^3 \setminus B_R)}\left\||w_n|^2-|w_0|^2\right\|_{L^{\frac{6}{5}}(\R^3 \setminus B_R)} \\
   \leq &C\left\||\hat{u}_n|^2-|u_{\infty}|^2\right\|_{L^{\frac{6}{5}}(\R^3)}+C\left\||w_n|^2-|w_0|^2\right\|_{L^{\frac{6}{5}}(B_R)}+\left\|\Phi_{|u_{\infty}|}\right\|_{L^6(\R^3 \setminus B_R)}\\
    =&o_n(1).
\end{align*}
Similarly, according to the Fubini theorem, we obtain
\begin{align*} &\left|\int_{\R^3}\int_{\R^3}\frac{\mathrm{Re}\left[(\hat{u}_n(y)+u_{\infty}(y))\overline{w_n}(y)\right]}{|x-y|}dyu_{\infty}(x)\mathrm{Re}(w_n(x))dx-\int_{\R^3}\int_{\R^3}\frac{2u_{\infty}(y)\mathrm{Re}\left(w_0(y)\right)}{|x-y|}dyu_{\infty}(x)\mathrm{Re}(w_0(x))dx\right|\\
\leq&\left|\int_{\R^3}\int_{\R^3}\frac{\mathrm{Re}\left[(\hat{u}_n(y)-u_{\infty}(y))\overline{w_n}(y)\right]}{|x-y|}dyu_{\infty}(x)\mathrm{Re}(w_n(x))dx\right|\\
&+\left|\int_{\R^3}\int_{\R^3}\frac{2u_{\infty}(y)\mathrm{Re}\left(w_n(y)-w_0(y)\right)}{|x-y|}dyu_{\infty}(x)\mathrm{Re}(w_n(x))dx\right|\\
&+\left|\int_{\R^3}\int_{\R^3}\frac{2u_{\infty}(y)\mathrm{Re}\left(w_0(y)\right)}{|x-y|}dyu_{\infty}(x)\mathrm{Re}(w_n(x)-w_0(x))dx\right|\\
    \leq&\int_{\R^3}\int_{\R^3}\frac{\left|(\hat{u}_n(y)-u_{\infty}(y))w_n(y)\right|}{|x-y|}dyu_{\infty}(x)|w_n(x)|dx+2\int_{\R^3}\int_{\R^3}\frac{u_{\infty}(y)|w_n(y)|
    }{|x-y|}dxu_{\infty}(x)|w_n(x)-w_0(x)|dy\\
    &+2\int_{\R^3}\int_{\R^3}\frac{u_{\infty}(y)|w_0(y)|}{|x-y|}dyu_{\infty}(x)|w_n(x)-w_0(x)|dx\\
    \leq&C\int_{\R^3}\Phi_{|\hat{u}_n-u_{\infty}|}^{\frac{1}{2}}\Phi_{|w_n|}^{\frac{1}{2}}u_{\infty}|w_n|dx+C\int_{\R^3}\Phi_{|u_{\infty}|}^{\frac{1}{2}}\Phi_{|w_n|}^{\frac{1}{2}}u_{\infty}|w_n-w_0|dx+C\int_{\R^3}\Phi_{|u_{\infty}|}^{\frac{1}{2}}\Phi_{|w_0|}^{\frac{1}{2}}u_{\infty}|w_n-w_0|dx\\
    =&o_n(1).
\end{align*}
Thus, $\|w_0\|=\lim\limits_{n\rightarrow\infty}\|w_n\|=1$. This is a contradiction with $w_0=0$.
\qed

\section{Proof of Theorem \ref{th1.4}}\label{s2}
In this section, we provide the detailed proof of Theorem \ref{th1.4}. 
First, we establish the equivalence between  $\|u\|_A$ and $\|u\|$.%, which is given as follows.%_{H^1(\R^3)}$.
\begin{lem}\label{A}
Suppose $(A)$ and $(V)$ hold, then $H^1_A(\mathbb{R}^3,\C)=H^1(\mathbb{R}^3,\C)$. Moreover, there exist $C_2>C_1>0$ such that 
\begin{align}\label{A00}
    C_1\|u\|\leq \|u\|_A\leq C_2\|u\|\ \ \text{for all}\ \ u\in H^1_A(\mathbb{R}^3,\C).
\end{align}
\end{lem}
\begin{proof}
By $(V)$ and $(A)$, we know that $A$ and $V$ are bounded in $\R^3$.
    Denote $V_1:=\sup_{x\in \mathbb{R}^3}V(x)\geq V_0$ and $A_0:=\sup_{x\in \mathbb{R}^3}|A(x)|$.
   Note that for all $u\in L^2(\R^3,\mathbb{C})$, it holds 
\begin{align}\label{A01}
    V_0\|u\|_{L^2(\R^3)}^2\le \int_{\R^3}V(x)|u|^2dx\le V_1\|u\|_{L^2(\R^3)}^2.
\end{align}
On the one hand, for all $u\in H_A^1(\R^3,\mathbb{C})$, we can prove that there exists $C_1>0$ such that 
\begin{align}
\label{A02}
    C_1\|u\|\leq \|u\|_A.
   % C_3\|\nabla u\|_{L^2(\R^3)}\le \|\nabla_A u\|_{L^2(\R^3)}.
    %\le C_4\|\nabla u\|_{L^2(\R^3)}.
\end{align}
In fact, 
\begin{align}\label{A03}
    &\|\nabla u\|_{L^2(\R^3)}^2=\int_{\R^3}|-i\nabla u|^2dx
    =\int_{\R^3}|-i\nabla u+A(x)u-A(x)u|^2dx\nonumber\\
    \le &2\int_{\R^3}|-i\nabla u+A(x)u|^2dx+2\int_{\R^3}|A(x)u|^2dx
    \le 2\|\nabla_Au\|_{L^2(\R^3)}^2+2A_0^2\|u\|_{L^2(\R^3)}^2.
\end{align}
So, combining \eqref{A01} with \eqref{A03}, we obtain \eqref{A02}, which implies $H_A^1(\R^3,\mathbb{C})\subset H^1(\R^3,\mathbb{C)}$. 

On the other hand, for all $u\in H^1(\R^3,\mathbb{C})$, we can also prove that there exists $C_2>0$ such that 
\begin{align}
\label{A04}
   \|u\|_A\leq C_2\|u\|.
\end{align}
Indeed, % Similar to \eqref{A03}, it holds that
\[
\|\nabla_A u\|_{L^2(\R^3)}^2
\leq \int_{\R^3}(|\nabla u|+|A(x)u|)^2dx
\le 2\|\nabla u\|_{L^2(\R^3)}^2+2A_0^2\|u\|_{L^2(\R^3)}^2,
\]
which, together with \eqref{A01}, yields \eqref{A04}. So,  $H^1(\R^3,\mathbb{C)}\subset H_A^1(\R^3,\mathbb{C)}$. Hence, $H^1_A(\mathbb{R}^3,\C)=H^1(\mathbb{R}^3,\C)$ and \eqref{A00} holds.
\end{proof}

\begin{rem}
    In general, the spaces $H^1_A(\R^3,\C)$ and $H^1(\R^3,\C)$ are incomparable. More precisely, $H^1_A(\R^3,\C)\not \subset H^1(\R^3,\C)$  and   $H^1(\R^3,\C)\not \subset H^1_A(\R^3,\C) $. If, in addition, $A$ and $V$ are bounded and $\inf_{x\in\R^3} V(x) > 0$, they are equivalent.
\end{rem}

Next,  we are going to prove Theorem \ref{th1.4}. Similar to Lemma \ref{lm2.0.1}, we have the following lemma.
%know that , and we will prove some properties of $c$.
%Firstly, we show that the Nehari manifold $\mathcal{N}$ defined by \eqref{1.12} is non-empty.
 \begin{lem}\label{lm2.1}
    Suppose $(A)$ and $(V)$ hold. For $v\in H^1_A(\mathbb{R}^3,\mathbb{C})\setminus \{0\}$, there exists a unique $t_v>0$ such that $t_vv\in \mathcal{N}$. Moreover, $t_v\in (0,1)$ if $J(v)<0$, while $t_v>1$ if $J(v)>0$.
%      %Especially, $\mathcal{N}$ is non-empty.
 \end{lem}
% \noindent{\bf Proof.}
% The proof is similar to Lemma \ref{lm2.0.1}, we omit the details.
% For each $v\in H_A^1(\mathbb{R}^3,\mathbb{C})\setminus \{0\}$ and $t\in \R^+$, define
% \begin{align*}
%         f_v(t):=I(tv)=\frac{t^2}{2}\|v\|_A^2-\frac{t^4}{4}\int_{\R^3}\Phi_v |v|^2 dx.
% \end{align*}
% One has
% \begin{align*}
%         f'_v(t)=t\left(\|v\|_A^2-t^2\int_{\R^3}\Phi_v |v|^2 dx\right),
% \end{align*}
% and then $f_v$ has a unique critical point 
% \begin{align}\label{N1}
%     t_v=\|v\|_A\left(\int_{\R^3}\Phi_v |v|^2 dx\right)^{-\frac{1}{2}}.
% \end{align}
% Obviously, $t_vv\in \mathcal{N}$. Moreover, $t_v\in (0,1)$ if $J(v)<0$, while $t_v>1$ if $J(v)>0$.
%\qed
% \begin{rem}\label{r1}
%     Similar to Lemma \ref{lm3.1}, we know for all $u\in H_A^1(\R^3,\mathbb{C})$ satisfies $J_{\infty}(u)<0$, there exists a unique $t\in(0,1)$ such that $tu\in \mathcal{N}_{\infty}$.
% \end{rem}
%By Lemma \ref{lm2.1}, 
% Hence, the ground state energy $c$ defined by \eqref{1.11} is well-defined, and we will prove some properties of $c$.

%The following lemmas show some properties of $c$.
So, the Nehari manifold $\mathcal{N}$ defined by \eqref{1.12} is non-empty and hence the ground state energy $c$ defined by \eqref{1.11} is well-defined.
\begin{lem}\label{lm3.1}
    Suppose $(A)$ and $(V)$ hold, then $c>0$. Moreover, if $u\in \mathcal{N}$ is a minimizer of $c$, then $u$ is a ground state solution of equation \eqref{eqs1.1.0}.
\end{lem}
\noindent{\bf Proof.}
For all $u\in \mathcal{N}$, by the Hardy-Littlewood-Sobolev inequality \eqref{HLS2} and the Sobolev inequality, we have
\begin{align*}%\label{n}
    \|u\|_A^2=\int_{\R^3}\Phi_{|u|} |u|^2dx\le C\| u \|_{L^{\frac{12}{5}}(\R^3)}^4\le C\|u\|_A^4.
\end{align*}
So, there exists $\rho>0$ such that $\|u\|_A\ge \rho$ for all $u\in \mathcal{N}$. Thus, 
$$I(u)=I(u)-\frac{1}{4}J(u)=\frac{1}{4}\|u\|_A^2\ge \frac{1}{4}\rho^2, \ \text{ for all } \   u\in \mathcal{N},$$
which implies that $c\ge \frac{1}{4}\rho^2>0$. 

Moreover, suppose that $u\in \mathcal{N}$ is a minimizer of $c$. It follows from the Lagrange multiplier method that there exists $\lambda \in \R$ such that 
\begin{align}\label{3.3.1}
    I'(u)- \lambda J'(u)=0.
\end{align}
We only need to prove $\lambda =0$. In fact, we deduce from \eqref{3.3.1} that 
\begin{align*}
  \lambda \langle J'(u),u\rangle 
  =\langle I'(u),u\rangle 
  =0.
\end{align*}
Notice that 
\begin{align*}
    \langle J'(u),u\rangle=2\|u\|_A^2-4\int_{\R^3}\Phi_{|u|} |u|^2dx
    =-2\|u\|_A^2<0,
\end{align*}
which gives $\lambda =0$. So, $u$ is a ground state solution of equation \eqref{eqs1.1.0}.
\qed

In the following paragraph, we study the relationship between the ground state energy of equation \eqref{eqs1.1.0} and that of equation \eqref{1.8.2}, which is vital to proving Theorem \ref{th1.4}. 
\begin{lem}\label{lm3.3}
      Suppose $(A)$ and $(V)$ hold. Then $c\leq c_{\infty}$. 
\end{lem}
\noindent{\bf Proof.} 
By Lemma \ref{lem:1}, $u_{\infty}$ is a ground state solution of equation \eqref{1.8.2}. %Obviously, since $(A)$ and $(V)$ hold, we know $u_{\infty}\in H_A^1(\R^3,\mathbb{C})$.  
Define 
$u_n(x):=u_{\infty}(x-y_n)$ with $|y_n|\to +\infty$, then there exists $t_n>0$ such that $t_nu_n\in \mathcal{N}$ for each $n\in \N$.
Moreover, %by \eqref{N1}, 
we know 
\begin{align*}
    t_n^2=\frac{\int_{\R^3}[|\nabla u_n|^2+(|A(x)|^2+V(x))|u_n|^2]dx}{\int_{\R^3}\Phi_{|u_n|} |u_n|^2dx}
    =\frac{\int_{\R^3}[|\nabla u_{\infty}|^2+(|A(x+y_n)|^2+V(x+y_n))u_{\infty}^2]dx}{\int_{\R^3}\Phi_{|u_{\infty}|} u_{\infty}^2dx}.
\end{align*}
For $R>0$, one has
\begin{align*}
    &\left|\int_{\R^3} [|A(x+y_n)|^2+V(x+y_n)-V_{\infty}]u_{\infty}^2dx\right|\\
    \leq &\left|\int_{\R^3\setminus{B_R}} [|A(x+y_n)|^2+V(x+y_n)-V_{\infty}]u_{\infty}^2dx\right|+\left|\int_{B_R} [|A(x+y_n)|^2+V(x+y_n)-V_{\infty}]u_{\infty}^2dx\right|\\
    \leq &C\int_{\R^3\setminus{B_R}}u_{\infty}^2dx+C\sup_{x\in B_R}\left||A(x+y_n)|^2+V(x+y_n)-V_{\infty}\right|.
\end{align*}
Then, according to $(A)$ and $(V)$, we obtain 
\begin{align*}
    \int_{\R^3} [|A(x+y_n)|^2+V(x+y_n)]u_{\infty}^2(x)dx
    =\int_{\R^3} V_{\infty}u_{\infty}^2dx+o_n(1).
\end{align*}
%for some $R>0$ big enough, one has
%\begin{align*}
%    &\int_{\R^3} [|A(x+y_n)|^2+V(x+y_n)]u_{\infty}^2(x)dx=\int_{B_R} (|A(x+y_n)|^2+V(x+y_n))u_{\infty}^2(x)dx\\
%    &+\int_{\R^3\setminus{B_R}} (|A(x+y_n)|^2+V(x+y_n))u_{\infty}^2(x)dx=\int_{\R^3} V_{\infty}u_{\infty}^2(x)dx+o(1).
%\end{align*}
Furthermore, 
\begin{align*}
    t_n^2=\frac{\int_{\R^3}[|\nabla u_{\infty}|^2+V_{\infty}u_{\infty}^2]dx+o_n(1)}{\int_{\R^3}\Phi_{|u_{\infty}|} u_{\infty}^2dx}=1+o_n(1). 
\end{align*}
So, 
\begin{align*}
    c\le I(t_nu_n)=\frac{t_n^2}{4}\|u_n\|_A^2= \frac{1+o_n(1)}{4}\left(\int_{\R^3}[\nabla u_{\infty}|^2+V_{\infty}u_{\infty}^2]dx+o_n(1)\right),
\end{align*}
which implies %$c\le \frac{1}{4}\|u_{\infty}\|^2=c_{\infty}$.
\begin{align}\label{3.3}
    c\le \lim_{n\to +\infty}I(t_nu_n)=\frac{1}{4}\|u_{\infty}\|^2=c_{\infty}.
\end{align}
\qed

%\begin{rem}\label{re2.1}
  %      If $v\in H^1(\mathbb{R}^3)$ satisfies $J_{\infty}(v):=\langle I'_0(v),v\rangle<0$, similar to Lemma \ref{lm2.1}, we can deduce that there exists a unique $t\in (0,1)$ such that $tv\in \mathcal{N}_{\infty}$.
%\end{rem}

Now, we give Br\'ezis-Lieb type splitting properties of the functional $\int_{\mathbb{R}^3}\Phi_{|u|} |u|^2 dx$ in $H^1(\mathbb{R}^3,\mathbb{C})$, which are vital to establishing the global compactness lemma. 
\begin{lem}%\cite[Lemma 2.2]{SM}
\label{lm2.3.0}
        Assume that $\{u_n\}$ is bounded in $H^1(\mathbb{R}^3,\mathbb{C})$, and there exists $u\in H^1(\mathbb{R}^3,\mathbb{C})$ such that $u_n \to u$ almost everywhere on $\R^3$, then as $n\to \infty$,
\begin{align}\label{A1}
        &\int_{\mathbb{R}^3}\Phi_{|u_n-u|} |u_n-u|^2 dx=\int_{\mathbb{R}^3}\Phi_{|u_n|} |u_n|^2 dx-\int_{\mathbb{R}^3}\Phi_{|u|} |u|^2 dx+o_n(1);\\
        \label{A2}        
&\mathrm{Re}\left(\int_{\mathbb{R}^3}\left(\Phi_{|u_n|} u_n-\Phi_{|u|} u-\Phi_{|u_n-u|} (u_n-u)\right)\overline{\varphi} dx\right)=o_n(1)\|\varphi\|, \ \text{ for all } \ \varphi \in H^1(\mathbb{R}^3,\mathbb{C})).
\end{align}
\end{lem}

In order to prove Lemma \ref{lm2.3.0}, we introduce an easy variant of the classical Br\'ezis-Lieb lemma \cite[Theorem 1]{BL2}. 
\begin{lem}
\label{BL}%\cite[Theorem 2]{B} and \cite[Theorem 4.2.7]{W2}
Let $\Omega \subseteq \mathbb{R}^3$ be a domain, $r \in [1, \infty)$ 
and $\{u_n\}$ be a bounded sequence in $L^r(\Omega,\C)$. %for some $r \in [q, \infty)$. 
If $u_n \to u$ almost everywhere on $\Omega$ as $n \to \infty$, then for every $q \in [1, r]$,
\[
\lim_{n \to \infty}
\int_{\Omega}
\left|
|u_n|^q - |u_n - u|^q - |u|^q
\right|^{r/q} dx
= 0.
\]
\end{lem}
% Also recall that pointwise convergence of a bounded sequence implies weak convergence (see \cite[Proposition 5.4.7]{W2}).
% \begin{lem}
% \label{weak}
% Let $\Omega \subseteq \mathbb{R}^3$ be a domain, $q \in (1, \infty)$, and let $(u_n)_{n \in \mathbb{N}}$ be a bounded sequence in $L^q(\Omega)$. If $u_n \to u$ almost everywhere on $\Omega$ as $n \to \infty$, then
% \[
% u_n \rightharpoonup u \quad \text{weakly in } L^q(\Omega) \ \ \text{as} \ \ n \to \infty.
% \]
% \end{lem}
{\noindent{\bf Proof of Lemma \ref{lm2.3.0}.}}
The proof is inspired by \cite{MV,ZZ}.
We first prove \eqref{A1} for completeness. Define $G(x,y):=\frac{1}{4\pi|x-y|}$. Note that 
% \begin{align}
%     & \int_{\mathbb{R}^3}\Phi_{|u_n|} |u_n|^2 dx-\int_{\mathbb{R}^3}\Phi_{|u_n-u|} |u_n-u|^2dx\nonumber\\
%     =&\int_{\R^3}(I_2*|u_n|^2)|u_n|^2dx-\int_{\R^3}(I_2*|u_n-u|^2)|u_n-u|^2dx\nonumber\\
% \end{align}
\begin{align}\label{31}
   & \int_{\mathbb{R}^3}\Phi_{|u_n|} |u_n|^2 dx-\int_{\mathbb{R}^3}\Phi_{|u_n-u|} |u_n-u|^2dx\nonumber\\
   =&\int_{\R^3}\int_{\R^3}G(x,y)|u_n(y)|^2|u_n(x)|^2dxdy-\int_{\R^3}\int_{\R^3}G(x,y)|u_n(y)-u(y)|^2|u_n(x)-u(x)|^2dxdy\nonumber\\
  % =&\int_{\R^3}\int_{\R^3}G(x,y)\left(|u_n(y)|^2-|u_n(y)-u(y)|^2\right)\left(|u_n(x)|^2-|u_n(x)-u(x)|^2\right)dxdy\nonumber\\
  % &+2\int_{\R^3}\int_{\R^3}G(x,y)\left(|u_n(y)|^2-|u_n(y)-u(y)|^2\right)|u_n(x)-u(x)|^2dxdy\nonumber\\
=&\int_{\R^3}\int_{\R^3}G(x,y)e_n(y)e_n(x)dxdy+2\int_{\R^3}\int_{\R^3}G(x,y)e_n(y)|u_n(x)-u(x)|^2dxdy,
\end{align}
where $e_n(y):=|u_n(y)|^2-|u_n(y)-u(y)|^2$. 
Since $\{u_n\}$ is bounded in $H^1(\mathbb{R}^3,\mathbb{C})$ and $u_n \to u$ almost everywhere on $\R^3$, using Lemma \ref{BL} with $q=2$ and $r=\frac{12}{5}$, we have $e_n-|u|^2\to 0$ in $L^{\frac{6}{5}}(\R^3,\C)$. 
Then, by the Hardy-Littlewood-Sobolev inequality \eqref{HLS2} and the H{\"o}lder inequality, it holds 
%\begin{align}
    %&\int_{\R^3}\int_{\R^3}G(x,y)\left(|u_n(y)|^2-|u_n(y)-u(y)|^2\right)\left(|u_n(x)|^2-|u_n(x)-u(x)|^2\right)dxdy-\int_{\R^3}\int_{\R^3}G(x,y)|u(y)|^2|u(x)|^2dxdy\nonumber\\
    % =&\int_{\R^3}\int_{\R^3}G(x,y)\left(|u_n(y)|^2-|u_n(y)-u(y)|^2-|u(y)|^2\right)\left(|u_n(x)|^2-|u_n(x)-u(x)|^2\right)dxdy\nonumber\\
    % &+\int_{\R^3}\int_{\R^3}G(x,y)|u(y)|^2\left(|u_n(x)|^2-|u_n(x)-u(x)|^2-|u(x)|^2\right)dxdy\nonumber\\
    % \le&\left(\int_{\R^3}\left(\int_{\R^3}G(x,y)\left(|u_n(y)|^2-|u_n(y)-u(y)|^2-|u(y)|^2\right)dy\right)^{6}dx\right)^{\frac{1}{6}}\|\left(|u_n|^2-|u_n-u|^2\right)\|_{L^{\frac{6}{5}}(\R^3)}\nonumber\\
    % &+\left(\int_{\R^3}\left(\int_{\R^3}G(x,y)|u(y)|^2dy\right)^{6}dx\right)^{\frac{1}{6}}\|\left(|u_n|^2-|u_n-u|^2-|u|^2\right)\|_{L^{\frac{6}{5}}(\R^3)}\nonumber\\
    % \le&C\|\left(|u_n|^2-|u_n-u|^2-|u|^2\right)\|_{L^{\frac{6}{5}}(\R^3)}+o_n(1)=o_n(1).
%\end{align}
\begin{align}\label{32}
    &\left|\int_{\R^3}\int_{\R^3}G(x,y)e_n(x)e_n(y)dxdy-\int_{\R^3}\int_{\R^3}G(x,y)|u(y)|^2|u(x)|^2dxdy\right|\nonumber\\
    \le &\int_{\R^3}\int_{\R^3}G(x,y)|e_n(x)|\cdot \left|e_n(y)-|u(y)|^2\right|dxdy+\int_{\R^3}\int_{\R^3}G(x,y)|u(y)|^2\left|e_n(x)-|u(x)|^2\right|dxdy\nonumber\\
     %\le&\left(\int_{\R^3}\left(\int_{\R^3}G(x,y)\left(e_n(y)-|u(y)|^2\right)dy\right)^{6}dx\right)^{\frac{1}{6}}\|e_n\|_{L^{\frac{6}{5}}(\R^3)}\nonumber\\
%&+\left(\int_{\R^3}\left(\int_{\R^3}G(x,y)|u(y)|^2dy\right)^{6}dx\right)^{\frac{1}{6}}\|(e_n-|u|^2)\|_{L^{\frac{6}{5}}(\R^3)}\nonumber\\
      \le&C\left\|e_n-|u|^2\right\|_{L^{\frac{6}{5}}(\R^3)}\|e_n\|_{L^{\frac{6}{5}}(\R^3)}+C\left\|e_n-|u|^2\right\|_{L^{\frac{6}{5}}(\R^3)}\|u\|_{L^{\frac{12}{5}}(\R^3)}^2=o_n(1).
\end{align}
% Introduce the Hardy-Littlewood-Sobolev inequality \eqref{HLS2} \cite[Thorem $A^*$]{SW} that if $u\in L^p(\R^N)$ for $1<p<q<\infty$, and $\frac{1}{q}=\frac{1}{p}+\frac{\lambda}{N}-1$ with $\lambda\in(0,N)$, then there exists $C_{p,q}$ is indepedent of $u$, but may depend on $p$ and $q$ such that 
% \[
% \left\|\int_{\R^N}\frac{u(y)}{|x-y|^{\lambda}}dy\right\|_{L^q(\R^N)}\le C_{p,q}\|u\|_{L^p(\R^N)}.
% \]
Moreover, it follows from %$e_n\to |u|^2$ in $L^{\frac{6}{5}}(\R^3)$ and
the Hardy-Littlewood-Sobolev inequality
\eqref{HLS1} that
\begin{align}\label{3.13}
\left\|\int_{\R^3}G(x,y)\left(e_n(y)-|u(y)|^2\right)dy\right\|_{L^6(\R^3)}
\leq C\left\|e_n-|u|^2\right\|_{L^{\frac{6}{5}}(\R^3)}
\to 0.
    %\int_{\R^3}G(x,y)e_n(y)dy\to \int_{\R^3}G(x,y)|u(y)|^2dy\ \text{in}\ \ L^6(\R^3).
\end{align}
According to the H{\"o}lder inequality and \eqref{3.13}, for $R>0$ large enough, we have
\begin{align}
    &\left|\int_{\R^3}\int_{\R^3}G(x,y)e_n(y)|u_n(x)-u(x)|^2dxdy\right|\nonumber\\
    \leq &\left|\int_{\R^3}\int_{\R^3}G(x,y)\left(e_n(y)-|u(y)|^2\right)|u_n(x)-u(x)|^2dxdy\right|
    +\left|\int_{\R^3}\int_{\R^3}G(x,y)|u(y)|^2|u_n(x)-u(x)|^2dxdy\right|\nonumber\\
    \leq &\left\|\int_{\R^3}G(x,y)\left(e_n(y)-|u(y)|^2\right)dy\right\|_{L^6(\R^3)}\|u_n-u\|_{L^{\frac{12}{5}}(\R^3)}^2
+\left|\int_{B_R}\left(\int_{\R^3}G(x,y)|u(y)|^2dy\right)|u_n(x)-u(x)|^2dx\right|\nonumber\\
&+\left|\int_{\R^3\setminus{B_R}}\left(\int_{\R^3}G(x,y)|u(y)|^2dy\right)|u_n(x)-u(x)|^2dx\right|\nonumber\\
\leq &o_n(1)+\left\|\int_{\R^3}G(x,y)|u(y)|^2dy\right\|_{L^6(\R^3)}\|u_n-u\|_{L^{\frac{12}{5}}(B_R)}^2
+\left\|\int_{\R^3}G(x,y)|u(y)|^2dy\right\|_{L^6(\R^3\setminus{B_R})}\|u_n-u\|_{L^{\frac{12}{5}}(\R^3)}^2\nonumber\\
\leq &o_n(1)+C\|u_n-u\|_{L^{\frac{12}{5}}(B_R)}^2+C\left\|\int_{\R^3}G(x,y)|u(y)|^2dy\right\|_{L^6(\R^3\setminus{B_R})},\nonumber
\end{align}
which, together with $u_n\to u$ strongly in $L_{loc}^{\frac{12}{5}}(\R^3,\mathbb{C})$, gives
\begin{align}\label{33}
    \left|\int_{\R^3}\int_{\R^3}G(x,y)e_n(y)|u_n(x)-u(x)|^2dxdy\right|=o_n(1).
\end{align}
% Moreover, combining $u_n\to u$ strongly in $L_{loc}^{\frac{12}{5}}(\R^3)$ with the Hardy-Littlewood-Sobolev inequality \eqref{HLS2} , for $R\in \R$ large enough, one has
% \begin{align}\label{33}
%    & \left|\int_{\R^3}\int_{\R^3}G(x,y)e_n(y)|u_n(x)-u(x)|^2dxdy\right|\nonumber\\
%    \le &\int_{B_R}\left(\int_{\R^3}G(x,y)|e_n(y)|dy\right)|u_n(x)-u(x)|^2dx+\int_{\R^3\setminus{B_R}}\left(\int_{\R^3}G(x,y)|e_n(y)|dy\right)|u_n(x)-u(x)|^2dx\nonumber\\
%    \le&C\left(\int_{B_R}\left(\int_{\R^3}G(x,y)|e_n(y)|dy\right)^{6}dx\right)^{\frac{1}{6}}\|u_n-u\|_{L^{\frac{12}{5}}(B_R)}^2\nonumber\\
%    &+C\|u_n-u\|_{L^{\frac{12}{5}}(\R^3)}^2\left(\int_{\R^3\setminus{B_R}}\left(\int_{\R^3}G(x,y)|e_n(y)|dy\right)^6dx\right)^{\frac{1}{6}}\nonumber\\
%    \le&C\|e_n\|_{L^{\frac{6}{5}}(\R^3)}\|u_n-u\|_{L^{\frac{12}{5}}(B_R)}^2+o_n(1)=o_n(1).
% \end{align}
%where we use
Hence, combining \eqref{31}, \eqref{32} and \eqref{33}, we derive \eqref{A1}.
%then $|u_n|$ is bounded in $L^q(\R^3)$ for $q\in[2,6]$ and $|u_n| \to |u|$ almost everywhere on $\R^3$. 

Next, we are going to prove \eqref{A2}.
    % The proof is similar to \cite[Lemma 2.2]{ZZ}, and \eqref{A1} has been proven in \cite[Lemma 2.2]{CW}, we only prove \eqref{A2}.  
    For each $\varphi \in H^1(\mathbb{R}^3,\mathbb{C}))$,
    we denote
    \[
    I_n:=\mathrm{Re}\left(\int_{\R^3}\Phi_{|u_n-u|}(u_n-u)\overline{\varphi}dx\right)=\mathrm{Re}\left(\int_{\R^3}\int_{\R^3}|u_n(y)-u(y)|^2(u_n(x)-u(x))\overline{\varphi}(x)G(x,y)dxdy\right),
    \]
    and $$I_0:=\mathrm{Re}\left(\int_{\R^3}\Phi_{|u|} u \overline{\varphi}dx \right)=\mathrm{Re}\left(\int_{\R^3}\int_{\R^3}|u(y)|^2u(x)\overline{\varphi}(x)G(x,y)dxdy\right).$$ %where $G(x,y)=\frac{1}{4\pi|x-y|}$.
   Notice that
    \[
    |u_n(y)-u(y)|^2=|u_n(y)|^2+|u(y)|^2-2\mathrm{Re}\left(u_n(y)\overline{u}(y)\right).
    \]
    Then,
\begin{align}       \label{A0}
I_n=&\mathrm{Re}\left(\int_{\R^3}\int_{\R^3}|u_n(y)|^2u_n(x)\overline{\varphi}(x)G(x,y)dxdy-\int_{\R^3}\int_{\R^3}|u(y)|^2u(x)\overline{\varphi}(x)G(x,y)dxdy\right)\nonumber\\
&+\mathrm{Re}\left(\int_{\R^3}\int_{\R^3}|u(y)|^2u_n(x)\overline{\varphi}(x)G(x,y)dxdy-\int_{\R^3}\int_{\R^3}|u_n(y)|^2u(x)\overline{\varphi}(x)G(x,y)dxdy\right)\nonumber\\
&+2\mathrm{Re}\left(\int_{\R^3}\int_{\R^3}\mathrm{Re}\left(u_n(y)\overline{u}(y)\right)u(x)\overline{\varphi}(x)G(x,y)dxdy\right)\nonumber\\
&-2\mathrm{Re}\left(\int_{\R^3}\int_{\R^3}\mathrm{Re}\left(u_n(y)\overline{u}(y)\right)u_n(x)\overline{\varphi}(x)G(x,y)dxdy\right)\nonumber\\
    =&\mathrm{Re}\left(\int_{\R^3}\Phi_{|u_n|} u_n\overline{\varphi}dx\right)-\mathrm{Re}\left(\int_{\R^3}\Phi_{|u|} u\overline{\varphi}dx\right)+I_n^{(1)}-I_n^{(2)}+2I_n^{(3)}-2I_n^{(4)},
    \end{align}
    where 
    \begin{align*}
        &I_n^{(1)}:=\mathrm{Re}\left(\int_{\R^3}\int_{\R^3}|u(y)|^2u_n(x)\overline{\varphi}(x)G(x,y)dxdy\right), \\ &I_n^{(2)}:=\mathrm{Re}\left(\int_{\R^3}\int_{\R^3}|u_n(y)|^2u(x)\overline{\varphi}(x)G(x,y)dxdy\right) , \\
    &I_n^{(3)}:=\mathrm{Re}\left(\int_{\R^3}\int_{\R^3}\mathrm{Re}\left(u_n(y)\overline{u}(y)\right)u(x)\overline{\varphi}(x)G(x,y)dxdy\right),\\ & I_n^{(4)}:=\mathrm{Re}\left(\int_{\R^3}\int_{\R^3}\mathrm{Re}\left(u_n(y)\overline{u}(y)\right)u_n(x)\overline{\varphi}(x)G(x,y)dxdy\right).
    \end{align*}
   Note that
      \begin{align}  \label{A5}
     & \left|I_n^{(4)}-I_0\right|=\left|\mathrm{Re}\left(\int_{\R^3}\int_{\R^3} \left[\mathrm{Re}\left(u_n(y)\overline{u}(y)\right)u_n(x)-|u(y)|^2u(x)\right]\overline{\varphi}(x)G(x,y)dxdy\right)\right|\nonumber\\
     \le &\left|\mathrm{Re}\left(\int_{\R^3}\int_{\R^3} \left[\mathrm{Re}\left(u_n(y)\overline{u}(y)\right)-|u(y)|^2\right]u_n(x)\overline{\varphi}(x)G(x,y)dxdy\right)\right|\nonumber\\
     &+
     \left|\mathrm{Re}\left(\int_{\R^3}\int_{\R^3} |u(y)|^2\left(u_n(x)-u(x)\right)\overline{\varphi}(x)G(x,y)dxdy\right)\right|\nonumber\\
     %=&o_n(1)\|\varphi\|_A.
     % =&\left|\mathrm{Re}\left(\int_{\R^3}\int_{\R^3} \mathrm{Re}\left[\big(u_n(y)-u(y)\big)\overline{u}(y)\right]u_n(x)\overline{\varphi}(x)G(x,y)dxdy\right)\right|+o_n(1)\|\varphi\|_A\nonumber\\
      \le&\int_{\R^3}\int_{\R^3} |u_n(y)-u(y)|\cdot |u(y)|\cdot |u_n(x)|\cdot |\varphi(x)|G(x,y)dxdy\nonumber\\
      &+\int_{\R^3}\int_{\R^3} |u(y)|^2|u_n(x)-u(x)|\cdot|\varphi(x)|G(x,y)dxdy.%+o_n(1)\|\varphi\|_A=o_n(1)\|\varphi\|_A.
%      =&\left|\mathrm{Re}\left(\int_{\R^3}\int_{\R^3}\mathrm{Re}\big(u_n(y)u_n(x)-u(y)u(x)\big)\overline{u}(y)\cdot\overline{\varphi}(x)G(x,y)dxdxy\right)\right|\nonumber\\
% \le&\int_{\R^3}\int_{\R^3}\big|u_n(y)u_n(x)-u(y)u(x)\big|\cdot |\overline{u}(y)|\cdot|\overline{\varphi}(x)|G(x,y)dxdxy=o_n(1)\|\varphi\|_A.
   \end{align}
%Since $u_n\rightharpoonup u$ in $H_A^1(\mathbb{R}^3,\mathbb{C})$, 
On the one hand, it follows from the H\"older inequality and the Minkowski inequality \cite[Theorem 2.4]{LL} that
    %by the proof of \cite[Lemma 2.2]{ZZ},it holds that 
   \begin{align}\label{A3}
%&\mathrm{Re}\left(\int_{\R^3}\int_{\R^3}|u(y)|^2u_n(x)\overline{\varphi}(x)G(x,y)dxdy\right)-\mathrm{Re}\left(\int_{\R^3}\int_{\R^3}|u(y)|^2u(x)\overline{\varphi}(x)G(x,y)dxdy\right)\\
%&|\widetilde{I_n^{(1)}}-I_0|=\left|\mathrm{Re}\left(\int_{\R^3}\int_{\R^3} \big(u_n(x)-u(x)\big)|u(y)|^2\overline{\varphi}(x)G(x,y)dxdy\right)\right|\nonumber\\
%\le& 
&\int_{\R^3}\int_{\R^3}|u_n(x)-u(x)|\cdot |u(y)|^2\cdot |\varphi(x)|G(x,y)dydx\nonumber\\
\le &\|\varphi \|_{L^3(\R^3)}\left(\int_{\R^3}\left(\int_{\R^3}|u_n(x)-u(x)|\cdot|u(y)|^2G(x,y)dy\right)^{\frac{3}{2}}dx\right)^{\frac{2}{3}}\nonumber\\
\le &C\|\varphi\|\int_{\R^3}\left(\int_{\R^3}\left(|u_n(x)-u(x)|\cdot|u(y)|^2G(x,y)\right)^{\frac{3}{2}}dx\right)^{\frac{2}{3}}dy\nonumber\\
=&C\|\varphi\|\int_{\R^3}\left(\int_{\R^3}|u_n(x)-u(x)|^{\frac{3}{2}}G^{\frac{3}{2}}(x,y)dx\right)^{\frac{2}{3}}|u(y)|^2dy\nonumber\\
=&C\|\varphi\|\int_{\R^3}z_n^{\frac{2}{3}}(y)|u(y)|^2dy,
   \end{align}
where $z_n(y):=\int_{\R^3}|u_n(x)-u(x)|^{\frac{3}{2}}G^{\frac{3}{2}}(x,y)dx$. For fixed $y\in \R^3$ and $R>0$ large enough, we have   
\begin{align}\label{A31}
    z_n(y)=&\int_{B(y,R)}|u_n(x)-u(x)|^{\frac{3}{2}}G^{\frac{3}{2}}(x,y)dx+\int_{\R^3\setminus B(y,R)}|u_n(x)-u(x)|^{\frac{3}{2}}G^{\frac{3}{2}}(x,y)dx\nonumber\\
    \le &\|u_n-u\|_{L^{\frac{9}{2}}(B(y,R))}^{\frac{3}{2}}\|G(x-y)\|_{L^{\frac{9}{4}}(B(y,R))}^{\frac{3}{2}}+\|u_n-u\|_{L^2(\R^3)}^{\frac{3}{2}}\|G(x-y)\|_{L^6(\R^3\setminus B(y,R))}^{\frac{3}{2}}\nonumber\\
    \leq &CR^{\frac{1}{2}}\|u_n-u\|_{L^{\frac{9}{2}}(B(y,R))}^{\frac{3}{2}}+CR^{-\frac{3}{4}}
    =o_n(1).
\end{align}
By the Hardy-Littlewood-Sobolev inequality \eqref{HLS1}, $\|z_n\|_{L^4(\R^3)}\le C\|u_n-u\|_{L^2(\R^3)}^{\frac{3}{2}}\le C$, which implies that $\{z_n^{\frac{2}{3}}\}$is bounded in $L^6(\R^3,\R)$.  Combining  \eqref{A31} and \cite[Proposition 5.4.7]{W2}, we obtain $z_n^{\frac{2}{3}}\rightharpoonup 0$ weakly in $L^6(\R^3,\R)$.
Hence, from \eqref{A3}, we have
\begin{align}\label{A32}
    &\int_{\R^3}\int_{\R^3}|u_n(x)-u(x)|\cdot |u(y)|^2\cdot |\varphi(x)|G(x,y)dxdy\le C\|\varphi\| \int_{\R^3}z_n^{\frac{2}{3}}(y)|u(y)|^2dy=o_n(1)\|\varphi\|.
\end{align}

On the other hand, since $u_n\to u$ strongly in $L^2_{loc}(\R^3,\C)$, for $R>0$ large enough, we deduce from the H\"older inequality and the Minkowski inequality that 
  \begin{align}\label{A61}
      &\int_{\R^3}\int_{\R^3} |u_n(y)-u(y)|\cdot |u(y)|\cdot |u_n(x)|\cdot |\varphi(x)|G(x,y)dydx \nonumber\\
      \le &\|\varphi\|_{L^3(\R^3)}\left(\int_{\R^3} \left( \int_{\R^3}|u_n(y)-u(y)|\cdot |u(y)|\cdot |u_n(x)|\cdot G(x,y)dy\right)^{\frac{3}{2}}dx\right)^{\frac{2}{3}}\nonumber\\
      %\le &C\|\varphi\|_A\int_{\R^3} \left(\int_{\R^3} |u_n(y)-u(y)|^{\frac{3}{2}}\cdot |u(y)|^{\frac{3}{2}}\cdot|u_n(x)|^{\frac{3}{2}} G^{\frac{3}{2}}(x,y)dx\right)^{\frac{2}{3}}dy\nonumber\\
      \le &C\|\varphi\|\int_{\R^3} \left(\int_{\R^3}|u_n(x)|^{\frac{3}{2}} G^{\frac{3}{2}}(x,y)dx\right)^{\frac{2}{3}}|u_n(y)-u(y)|\cdot |u(y)|dy\nonumber\\
      =&C\|\varphi\|\int_{\R^3}z_n^*(y)|u_n(y)-u(y)|\cdot |u(y)|dy\nonumber\\
      \le&C\|\varphi\|\cdot\|z_n^*\|_{L^6(\R^3)}\left(\|u_n-u\|_{L^2(B_R)}\|u\|_{L^3(\R^3)}+\|u_n-u\|_{L^2(\R^3)}\|u\|_{L^3(\R^3\setminus{B_R})}\right) \nonumber\\
      =&o_n(1)\|\varphi\|,
  \end{align}
  where $z_n^*(y)=\left(\int_{\R^3}|u_n(x)|^{\frac{3}{2}} G^{\frac{3}{2}}(x,y)dx\right)^{\frac{2}{3}}$ and $\|z_n^*\|_{L^6(\R^3)}\le C\|u_n\|_{L^2(\R^3)} \le C $.
    Then,  combining \eqref{A5}, \eqref{A32} and \eqref{A61}, we have
 \begin{align}
     \label{A7}
     \left|I_n^{(4)}-I_0\right|= o_n(1)\|\varphi\|.
 \end{align}  
 
 Similar to the proof of \eqref{A7}, we can show that the following equalities hold
 \begin{align}
     \label{A8}
    \left|I_n^{(1)}-I_0\right|= o_n(1)\|\varphi\|,\  \left|I_n^{(2)}-I_0\right|= o_n(1)\|\varphi\|,\ \text{and} \    \left|I_n^{(3)}-I_0\right|= o_n(1)\|\varphi\|.
 \end{align}
%    \begin{align}   \label{A6}
% &|I_n^{(2)}-I_0|=\left|\mathrm{Re}\left(\int_{\R^3}\int_{\R^3}\left(|u_n(y)|^2-|u(y)|^2\right)u(x)\overline{\varphi}(x)G(x,y)dxdy\right)\right|\nonumber\\
% \le & \int_{\R^3}\int_{\R^3}\big||u_n(y)|^2-|u(y)|^2\big|\cdot  |u(x)|\cdot|\varphi(x)|G(x,y)dxdy
% = o_n(1)\|\varphi \|_A,
%    \end{align}
%    and 
% \begin{align}
% \label{A4}
% &|\widetilde{I_n^{(3)}}-I_0|=\left|\mathrm{Re}\left(\int_{\R^3}\int_{\R^3}\mathrm{Re}\left(|u_n(y)|\overline{u}(y)-|u(y)|^2\right)u(x)\overline{\varphi}(x)G(x,y)dxdy\right)\right|\nonumber\\
% =&\left|\mathrm{Re}\left(\int_{\R^3}\int_{\R^3} \mathrm{Re}\left[\big(u_n(y)-u(y)\big)\overline{u}(y)\right]u(x)\overline{\varphi}(x)G(x,y)dxdy\right)\right|\nonumber\\
%      \le&\int_{\R^3}\int_{\R^3} |u_n(y)-u(y)|\cdot |\overline{u}(y)|\cdot |u(x)|\cdot |\overline{\varphi}(x)|G(x,y)dxdy =o_n(1)\|\varphi\|_A.
% \end{align} 
Hence, it follows from \eqref{A0}, \eqref{A7} and \eqref{A8} that
\begin{align*}
    &\left|I_n-\mathrm{Re}\left(\int_{\R^3}\Phi_{|u_n|} u_n\overline{\varphi}dx\right)+\mathrm{Re}\left(\int_{\R^3}\Phi_{|u|} u\overline{\varphi}dx\right)\right|\\
\le &\left| I_n^{(1)}-I_0\right|+\left|I_n^{(2)}-I_0\right|+2\left|I_n^{(3)}-I_0\right|+2\left| I_n^{(4)}-I_0\right|= o_n(1)\|\varphi\|.
\end{align*}
Therefore, \eqref{A2} holds. We complete the proof.
\qed

\vskip 0.08truein

Inspired by \cite{S}, we establish a global compactness lemma for a $(PS)$ sequence of $I$, which provides a complete description of this $(PS)$ sequence.
\begin{lem}\label{lm4.2}
        Suppose $(A)$ and $(V)$ hold. Let $\{u_n\}\subset  H^1_A(\R^3,\mathbb{C})$ be a sequence such that
        \begin{align}\label{4.1}
                I(u_n)\to d, \ I'(u_n)\to 0\ \text{ in }H^*,\  \text{as}\ n\to +\infty,
        \end{align}
        where $H^*$ is the dual space of $H_A^1(\mathbb{R}^3,\mathbb{C})$, whose norm is denoted by $\|\cdot\|_*$.
        Then, replacing $\{u_n\}$ if necessary by a subsequence, there exists a solution $u_0\in H^1_A(\R^3,\mathbb{C})$ of equation \eqref{eqs1.1.0}, a finite sequence of solutions $\{v_j\}_{j=1}^k\subset H^1(\mathbb{R}^3,\mathbb{C})\setminus \{0\}$ of equation \eqref{1.8.2}, and k sequences of points $\{z_n^j\}\subset \mathbb{R}^3$ satisfying

        $(i)$\, $|z_n^j|\to +\infty$, $|z_n^j-z_n^i|\to +\infty$, $j\neq i$, as $n\to +\infty$;

        $(ii)$\, $u_n-u_0-\sum_{j=1}^k v_j(\cdot -z_n^j)\to 0$ in $H_A^1(\mathbb{R}^3,\mathbb{C})$;

        $(iii)$\, $d=I(u_0)+\sum_{j=1}^k I_{\infty}(v_j)$.
\end{lem}
\noindent{\bf Proof.} For $n$ large enough, by \eqref{4.1}, we have 
\begin{align*}%\label{b}
        d+1+\|u_n\|_A\geq I(u_n)-\frac{1}{4}\left\langle I'(u_n),u_n \right\rangle
       =\frac{1}{4}\|u_n\|_A^2,
\end{align*}
%Similar to \eqref{4.0}, 
so $\{u_n\}$ is bounded in $H_A^1(\mathbb{R}^3,\mathbb{C})$.
% Therefore,
% \begin{align}\label{4.1.4}
%         o_n(1)= \left\langle I'_+(u_n),(u_n)_- \right\rangle =\|(u_n)_-\|^2-\int_{\R^3}\Phi_0 ((u_n)_-)^2 dx\geq \|(u_n)_-\|^2,
%  \end{align}
%  where $u_-:=\min \{u,0\}$.
%  Since $I_+\in C^1(H^1(\mathbb{R}^3),\R)$, by \eqref{4.1} and \eqref{4.1.4}, we have
%  \begin{align}
%         I((u_n)_+)= &I_+((u_n)_+)=I_+(u_n)+o_n(1)=d+o_n(1),\label{4.1.2}\\
%         \|I'((u_n)_+)\| _{H^{-1}( \mathbb{R} ^3)}=&\|I'_+((u_n)_+)\| _{H^{-1}( \mathbb{R} ^3)}=\|I'_+(u_n)\| _{H^{-1}( \mathbb{R} ^3)}+o_n(1)=o_n(1).\label{4.1.3}
%  \end{align}
%  Namely, $\{(u_n)_+\}$ is a nonnegative $(PS)_d$ sequence of $I$.
% Because $\{(u_n)_+\}$ is bounded in $H^1(\mathbb{R}^3)$, 
Then there exists $u_0\in H_A^1(\mathbb{R}^3,\mathbb{C})$ such that, up to a subsequence,
\begin{eqnarray}\label{4.1.1}
        \left\{\begin{array}{ll}
                u_n\rightharpoonup u_0 \ \ \text{weakly in}\  H_A^1(\mathbb{R}^3,\mathbb{C}),\\
                u_n\rightarrow u_0 \ \ \text{strongly in}\  L^p_{loc}(\mathbb{R} ^3,\mathbb{C}), \ p\in [2,6).\\
        \end{array} \right.	
\end{eqnarray}
Moreover, similar to \cite[Corollary 9]{B1}, we know that for each $\varphi\in H_A^1(\mathbb{R}^3,\mathbb{C})$, it holds  
\[
\lim_{n\to \infty}\mathrm{Re}\left(\int_{\R^3}\Phi_{|u_n|}u_n\bar{\varphi}dx\right)=\mathrm{}{Re}\left(\int_{\R^3}\Phi_{|u_0|}u_0\bar{\varphi}dx\right).
\]
So,
\begin{align*}
       \left\langle I'(u_0),\varphi \right\rangle =(u_0,\varphi )_A -\mathrm{Re}\left(\int_{\R^3}\Phi_{|u_0|} u_0\bar{\varphi} dx\right)=0.
\end{align*}
Hence, $I'(u_0)=0$ and $u_0$ is a solution of equation \eqref{eqs1.1.0}.

Denote $u_n^{(1)}:=u_n-u_0$. We are going to prove that 
\begin{align}\label{4.2}
        I_{\infty}(u_n^{(1)})=d-I(u_0)+o_n(1).\
        %\langle I'_{\infty}(u_n^{(1)}),\varphi\rangle =o_n(1) \text{ for each }\varphi \in H_A^1(\R^3,\mathbb{C}).
        %I'_{\infty}(u_n^{(1)})=o_n(1)\ \text{in} \ H^*,
\end{align}
%where $\varphi \in H_A^1(\R^3,\mathbb{C})$.
%where $H^*$ is the dual space of $H_A^1(\mathbb{R}^3,\mathbb{C})$, whose norm is denoted by $\|\cdot\|_*$.
%that is, $\{u_n^{(1)}\}$ is a $(PS)_{d-I(u_0)}$ sequence of $I_{\infty}$. 
In fact, 
\begin{align*}
    &\int_{\R^3}\left|\nabla u_n^{(1)}\right|^2dx
    =\int_{\R^3}\left|-i\nabla u_n^{(1)}\right|^2dx
    =\int_{\R^3}\left|\nabla_A u_n^{(1)}-A(x)u_n^{(1)}\right|^2dx\\
    =&\int_{\R^3}\left(\left|\nabla_A u_n^{(1)}\right|^2 +\left|A(x)u_n^{(1)}\right|^2 -2\mathrm{Re}\left(\nabla_A u_n^{(1)}A(x)\overline{u_n^{(1)}}\right)\right)dx.
\end{align*}
Since $u_n^{(1)}\to 0$ in  $ L^2_{loc}(\mathbb{R} ^3,\mathbb{C})$ as $n \to +\infty$, by $(A)$, we obtain 
\begin{align}\label{3.14}
    \int_{\R^3} \left|A(x)u_n^{(1)}\right|^2 dx=o_n(1).
\end{align}
According to the H{\"o}lder inequality, we get 
\begin{align*}
    \left|\int_{\R^3}\mathrm{Re}\left(\nabla_A u_n^{(1)}A(x)\overline{u_n^{(1)}}\right)dx\right|
    \leq \int_{\R^3}\left|\nabla_A u_n^{(1)}A(x)\overline{u_n^{(1)}}\right|dx
    \leq \|\nabla_A u_n^{(1)}\|_{L^2(\mathbb{R}^3)} \|A u_n^{(1)}\|_{L^2(\mathbb{R}^3)}=o_n(1).
\end{align*}
Thus, 
\begin{align}\label{4.3}
    \int_{\R^3}\left|\nabla u_n^{(1)}\right|^2dx
    =&\int_{\R^3}\left|\nabla_A u_n^{(1)}\right|^2dx+o_n(1)
    =\int_{\R^3}|\nabla_A u_n-\nabla_A u_0|^2dx+o_n(1)\nonumber\\
    =&\int_{\R^3}|\nabla_A u_n|^2dx+\int_{\R^3}|\nabla_A u_0|^2dx-2\mathrm{Re}\int_{\R^3}\nabla_A u_n\overline{\nabla_Au_0}\ dx+o_n(1)\nonumber\\
   =& \int_{\R^3}|\nabla_A u_n|^2dx-\int_{\R^3}|\nabla_A u_0|^2dx+o_n(1),
\end{align}
where we used \eqref{4.1.1}.

Similar to \eqref{3.14}, taking \eqref{4.1.1} into consideration, we have 
\begin{align}\label{3.16.1}
\int_{\R^3} V_{\infty}|u_n^{(1)}|^2dx
=&\int_{\R^3} V(x)|u_n^{(1)}|^2dx+o_n(1)
=\int_{\R^3} V(x)\left(|u_n|^2+|u_0|^2
-2\mathrm{Re}(u_n\overline{u_0})\right)dx+o_n(1)\nonumber\\
=&\int_{\R^3}V(x)|u_n|^2dx-\int_{\R^3}V(x)|u_0|^2dx+o_n(1).
\end{align}
From Lemma \ref{lm2.3.0}, we obtain
\begin{align}\label{3.11}
    \int_{\mathbb{R}^3}\Phi_{|u_n^{(1)}|} \left|u_n^{(1)}\right|^2 dx=\int_{\mathbb{R}^3}\Phi_{|u_n|} |u_n|^2 dx-\int_{\mathbb{R}^3}\Phi_{|u_0|} |u_0|^2 dx+o_n(1).
\end{align}
Then, combining \eqref{4.3}, \eqref{3.16.1} and \eqref{3.11}, we derive 
\begin{align}
     &I_{\infty}(u_n^{(1)})
     =I(u_n)-I(u_0)+o_n(1)
     =d-I(u_0)+o_n(1),\nonumber\\
     &\langle I'_{\infty}(u_n^{(1)}),u_n^{(1)}\rangle=\langle I'(u_n),u_n\rangle -\langle I'(u_0),u_0\rangle+o_n(1)
     =o_n(1).\label{L2}
\end{align}
% Moreover, since $u_n^{(1)} \rightharpoonup 0$ in $H_A^1(\R^3,\mathbb{C})$, combining \cite[Lemma 2.1]{CW}, we know that for each $\varphi \in H_A^1(\R^3,\mathbb{C})$, it holds
% \begin{align*}%\label{L3}
%     \langle I'(u_n^{(1)}),\varphi \rangle=(u_n^{(1)},\varphi) -\mathrm{Re}\left(\int_{\R^3}\Phi_{|u_n^{(1)}|} u_n^{(1)}\bar{\varphi} dx\right)=o_n(1).
% \end{align*}
Hence, \eqref{4.2} holds.

Let
\begin{align*}
        \delta_1:=\lim_{n\to +\infty}\sup_{z\in \R^3}\int_{B(z,1)}\left|u_n^{(1)}(x)\right|^2dx.
\end{align*}
If $\delta_1=0$, from the vanishing lemma, one has $|u_n^{(1)}|\rightarrow  0$ in $L^p(\mathbb{R} ^3,\mathbb{R})$ for $p\in (2,6)$. 
%So, $u_n^{(1)}\rightarrow  0$ in $L^p(\mathbb{R} ^3,\mathbb{C})$ for $p\in (2,6)$. 
%Since $\{u_n^{(1)}\}$ is bounded in $H_A^1(\mathbb{R}^3,\mathbb{C})$, 
By \eqref{L2} and the Hardy-Littlewood-Sobolev inequality \eqref{HLS2}, we obtain
\begin{align}\label{4.7}
       \|u_n^{(1)}\|^2
        =\left\langle I'_{\infty}(u_n^{(1)}),u_n^{(1)} \right\rangle+\int_{\R^3}\Phi_{|u_n^{(1)}|} \left|u_n^{(1)}\right|^2dx
        \leq o_n(1)+C \|u_n^{(1)}\|^4_{L^{\frac{12}{5}}(\mathbb{R}^3)}=o_n(1),
\end{align}
which yields that
\[
I_{\infty}(u_n^{(1)})=o_n(1).
\]
Hence, %$u_n\to u_0$ in $H_A^1(\R^3,\mathbb{C})$, $d=I(u_0)$, and $u_0$ is a solution of the equation \eqref{eqs1.1.0}.
combining \eqref{4.2}, we derive $d=I(u_0)$, and the conclusion is obtained.
%\begin{align*}
 %       d=I(u_0).
%\end{align*}
% Combining \eqref{4.1.4} and \eqref{4.7}, it holds
% \begin{align*}
%         \|u_n-u_0\|\leq \|(u_n)_+-u_0\|+\|(u_n)_-\|=o_n(1),
% \end{align*}
% hence the conclusion is obtained.

If $\delta_1>0$, there exists $\{z_n^1\}\subset \R^3$ such that
\begin{align*}
        \int_{B(z_n^1,1)}\left|u_n^{(1)}(x)\right|^2dx\geq \frac{\delta_1}{2}.
\end{align*}
Due to $|u_n^{(1)}|\to 0$ in $L^2_{loc}(\mathbb{R}^3,\R)$, we get $|z_n^1| \to +\infty$. Letting $v_n^{(1)}(x):=u_n^{(1)}(x+z_n^1)$, there exists $v_1\in H_A^1(\mathbb{R}^3,\mathbb{C})$ such that, up to a subsequence,
\begin{eqnarray}\label{4.9}
        \left\{\begin{array}{ll}
                v_n^{(1)}\rightharpoonup v_1 \ \ \text{weakly in}\  H_A^1(\mathbb{R}^3,\mathbb{C}),\\
                v_n^{(1)}\rightarrow v_1 \ \ \text{strongly in}\  L^p_{loc}(\mathbb{R} ^3,\mathbb{C}), \ p\in [2,6).\\
        \end{array} \right.	
\end{eqnarray}
Consequently,
\begin{align*}
        \int_{B(0,1)}|v_1(x)|^2dx=\lim_{n\to +\infty}\int_{B(0,1)}\left|v_n^{(1)}(x)\right|^2dx
        =\lim_{n\to +\infty}\int_{B(z_n^1,1)}\left|u_n^{(1)}(x)\right|^2dx\geq \frac{\delta_1}{2},
\end{align*}
which implies that $v_1\neq 0$.
We deduce from \eqref{4.2} that
\begin{align}
        I_{\infty}(v_n^{(1)})&= I_{\infty}(u_n^{(1)})=d-I(u_0)+o_n(1).\label{4.10}
\end{align}

Next, we are going to prove 
\begin{align}
    \langle I'_{\infty}(v_n^{(1)}),\varphi \rangle&=\langle I'_{\infty}(u_n^{(1)}), \varphi (x-z_n^1)\rangle=o_n(1),\ \text{ for any } \varphi\in H_A^1(\R^3,\mathbb{C}).\label{4.11}
\end{align}
Define $\varphi_n(x):=\varphi(x-z_n^1)$. 
By ($A$), for $R$ large enough, we have
\begin{align*}
    \|A\overline{ \varphi_n}\|_{L^2(\R^3)}^2=\int_{\R^3}|A(x+z_n^1)\varphi(x)|^2dx=\int_{B_R}|A(x+z_n^1)\varphi(x)|^2dx+\int_{\R^3\setminus B_R}|A(x+z_n^1)\varphi(x)|^2dx=o_n(1),
\end{align*}
and
\begin{align*}
   \|Au_n^{(1)}\|_{L^2(\R^3)}^2=\int_{B_R}|A u_n^{(1)}(x)|^2dx+\int_{\R^3\setminus B_R}|A u_n^{(1)}(x)|^2dx=o_n(1).
\end{align*}
Similarly, $\||A|^2u_n^{(1)}\|_{L^2(\R^3)}=o_n(1)$. Moreover, it follows from 
($V$) that 
\begin{align*}
    \mathrm{Re}\left(\int_{\R^3}(V_{\infty}-V)u_n^{(1)}\overline{\varphi_n}dx\right)=o_n(1).
\end{align*}
Recall that $I'(u_n)=o_n(1)$ and  $I'(u_0)=0$, we have
\[
\langle I'(u_n),\varphi_n\rangle =o_n(1),\ \ \text{and}\ \ \langle I'(u_0),\varphi_n\rangle =o_n(1).
\]
So, combining Lemma \ref{lm2.3.0}, it holds
\begin{align*}
   & \left|\langle I'_{\infty}(u_n^{(1)}), \varphi_n\rangle\right|
   \leq \left|\langle I'_{\infty}(u_n^{(1)}), \varphi_n\rangle-\langle I'(u_n),\varphi_n\rangle+\langle I'(u_0),\varphi_n\rangle \right|+o_n(1)\nonumber\\
  \leq &\left|\mathrm{Re}\left(\int_{\R^3}( \nabla u_n^{(1)}\overline{\nabla \varphi_n}-\nabla_A u_n^{(1)}\overline{\nabla_A \varphi_n})dx\right)\right|
  +\left|\mathrm{Re}\left(\int_{\R^3}(V_{\infty}-V)u_n^{(1)}\overline{\varphi_n}dx\right)\right|\nonumber\\
   &+\left|\mathrm{Re}\left(\int_{\R^3}(\Phi_{|u_n^{(1)}|}u_n^{(1)}-\Phi_{|u_n|}u_n-\Phi_{|u_0|}u_0)\overline{\varphi_n}dx\right)\right|+o_n(1)\nonumber\\
   =&\left|\mathrm{Re}\left(\int_{\R^3}(\nabla u_n^{(1)}\overline{\nabla \varphi_n}-\nabla_A u_n^{(1)}\overline{\nabla_A \varphi_n})dx\right)\right|+o_n(1)\nonumber\\
   \leq &\left|\mathrm{Re}\left(\int_{\R^3}i(\nabla u_n^{(1)}A\overline{ \varphi_n}-A u_n^{(1)}\overline{\nabla \varphi_n})dx\right)\right|
+\left|\mathrm{Re}\left(\int_{\R^3}|A|^2u_n^{(1)}\overline{\varphi_n}dx\right)\right|+o_n(1)\nonumber\\
   \le &\int_{\R^3}|\nabla u_n^{(1)}|\cdot|A \varphi_n|dx+\int_{\R^3}|A u_n^{(1)}|\cdot|\nabla \varphi_n|dx+\int_{\R^3}|A|^2|u_n^{(1)}|\cdot|\varphi_n|dx+o_n(1)\nonumber\\
   \le &\|\nabla u_n^{(1)}\|_{L^2(\R^3)}\|A\varphi_n\|_{L^2(\R^3)}+\|\nabla \varphi_n\|_{L^2(\R^3)}\|Au_n^{(1)}\|_{L^2(\R^3)}+\||A|^2u_n^{(1)}\|_{L^2(\R^3)}\|\varphi_n\|_{L^2(\R^3)}+o_n(1)\nonumber\\
   =&o_n(1).
\end{align*}
% where we used
% \begin{align*}
%     \|A\overline{ \varphi_n}\|_{L^2(\R^3)}^2=\int_{\R^3}|A(x+z_n^1)\varphi(x)|^2dx=\int_{B_R}|A(x+z_n^1)\varphi(x)|^2dx+\int_{\R^3\setminus B_R}|A(x+z_n^1)\varphi(x)|^2dx=o_n(1),
% \end{align*}
% and
% \begin{align*}
%    \|Au_n^{(1)}\|_{L^2(\R^3)}^2=\int_{B_R}|A u_n^{(1)}(x)|^2dx+\int_{\R^3\setminus B_R}|A u_n^{(1)}(x)|^2dx=o_n(1).
% \end{align*}
Then, combining \eqref{4.9} and \eqref{4.11}, we obtain $I'_{\infty}(v_1)=0$ and hence $v_1$ is a nontrivial solution of the equation \eqref{1.8.2}.

Define $u_n^{(2)}(x):=u_n^{(1)}(x)-v_1(x-z_n^1)$. %Similar to \eqref{L1},
Using \eqref{L2}, \eqref{4.10} and Lemma \ref{lm2.3.0}, we have
\begin{align}
        I_{\infty}(u_n^{(2)})=&I_{\infty}(v_n^{(1)}-v_1)=I_{\infty}(v_n^{(1)})-I_{\infty}(v_1)+o_n(1)=d-I(u_0)-I_{\infty}(v_1)+o_n(1),\label{4.13}\\
      % &I'_{\infty}(u_n^{(2)})=I'_{\infty}(v_n^{(1)}-v_1)=I'_{\infty}(v_n^{(1)})-I'_{\infty}(v_1)+o_n(1)=o_n(1) \text{in} H^*\label{4.131},
        \langle I'_{\infty}(u_n^{(2)}),u_n^{(2)}\rangle
        =&\langle I'_{\infty}(v_n^{(1)}-v_1), v_n^{(1)}-v_1\rangle
       =\langle I'_{\infty}(v_n^{(1)}),v_n^{(1)}\rangle+\langle I'_{\infty}(v_1),v_1\rangle+o_n(1)\nonumber\\
       =&\langle I'_{\infty}(u_n^{(1)}),u_n^{(1)}\rangle+o_n(1)
       =o_n(1).\nonumber
\end{align}
%where $H^*$ is the dual space of $H_A^1(\mathbb{R}^3,\mathbb{C})$, whose norm is denoted by $\|\cdot\|_*$.
% In fact, using \eqref{4.10}, \eqref{4.11} and Lemma \ref{lm2.3.0}, \eqref{4.13} is easy to obtain. Next, we prove \eqref{4.131}. Using Lemma \ref{lm2.3.0}, we have
 %  \begin{align*}%\label{4.5}
 %  &\|I'_{\infty}(u_n^{(2)})\|_*=\|I'_{\infty}(v_n^{(1)})-I'_{\infty}(v_1)\|_*=\sup_{\|\varphi \|_A=1}\left|\langle I'_{\infty}(v_n^{(1)}),\varphi \rangle -\langle I'_{\infty}(v_1),\varphi \rangle\right|\\
 %  &    
 % \end{align*}
% which, together with \eqref{4.1}, gives the second equality of \eqref{4.2}.
%that is, $\{u_n^{(2)}\}$ is a $(PS)_{d-I(u_0)-I_{\infty}(v_1)}$ sequence of $I_{\infty}$.
Denote
\begin{align*}
        \delta_2:=\lim_{n\to +\infty}\sup_{z\in \R^3}\int_{B(z,1)}\left|u_n^{(2)}(x)\right|^2dx.
\end{align*}
If $\delta_2=0$, according to the vanishing lemma, we get $|u_n^{(2)}|\rightarrow 0$ in $L^p(\mathbb{R} ^3,\R)$ for $p\in (2,6)$. Analogous to \eqref{4.7}, it holds that $\| u_n^{(2)}\| =o_n(1)$, which, together with \eqref{4.13}, gives $d=I(u_0)+I_{\infty}(v_1)$. %$\|\nabla u_n^{(2)}\|_{L^2(\R^3)}^2+V_{\infty}\|u_n^{(2)}\|_{L^2(\R^3)}^2=o_n(1)$.
Using Lemma \ref{A}, we have $\|u_n^{(2)}\|_A=o_n(1)$, that is, $u_n^{(2)}\rightarrow 0$ in $H_A^1(\mathbb{R}^3,\mathbb{C})$.
% \begin{align*}
%     &\|u_n^{(2)}\|_A^2=\|\nabla u_n^{(2)}\|_{L^2(\R^3)}^2+\int_{\R^3}(|A|^2+V(x))|u_n^{(2)}|^2dx+2\mathrm{Im}\left(\int_{\R^3}A\overline{u_n^{(2)}}\cdot \nabla u_n^{(2)}dx\right)\\
%     \leq &\int_{B_R}(|A|^2+V(x))|u_n^{(2)}|^2dx+\int_{\R^3\setminus{B_R}}(|A|^2+V(x))|u_n^{(2)}|^2dx+\|Au_n^{(2)}\|_{L^2(\R^3)}\|\nabla u_n^{(2)}\|_{L^2(\R^3)}+o_n(1)\\
%     \le&C \int_{B_R}|u_n^{(2)}|^2dx+\int_{\R^3\setminus{B_R}}V_{\infty}|u_n^{(2)}|^2dx+o_n(1)=o_n(1).
% \end{align*}
% Then we have
% \begin{align*}
%         \|u_n-u_0-v_1(\cdot -z_n^1)\|_A\leq %&\|(u_n)_+-u_0-v_1(\cdot -z_n^1)\|+\|(u_n)_-\|\\
%         \|u_n^{(2)}\|_A+\|(u_n)_-\|=o_n(1).
% \end{align*}
The conclusion is obtained.

If $\delta_2>0$, there exists $\{z_n^2\}\subset \R^3$ such that
\begin{align}\label{v2}
        \int_{B(z_n^2,1)}\left|u_n^{(2)}(x)\right|^2dx\geq \frac{\delta_2}{2}.
\end{align}
Note that
\begin{align*}
        \int_{B(z_n^2-z_n^1,1)}\left|v_n^{(1)}(y)-v_1(y)\right|^2dy
        =\int_{B(z_n^2,1)}\left|u_n^{(1)}(x)-v_1(x-z_n^1)\right|^2dx
        =\int_{B(z_n^2,1)}\left|u_n^{(2)}(x)\right|^2dx\geq \frac{\delta_2}{2},
\end{align*}
which, together with \eqref{4.9}, implies $|z_n^2-z_n^1| \to +\infty$. Next, we claim that $|z_n^2| \to +\infty$. In fact, if $|z_n^2|$ is bounded, there exists $R>0$ such that $B(z_n^2,1)\subset B_R$. Therefore, it follows from $|u_n^{(1)}|\to 0$ in $L^2_{loc}(\mathbb{R}^3,\mathbb{R})$ and $|z_n^1| \to +\infty$ that
\begin{align*}
        \frac{\delta_2}{2}\leq &\int_{B(z_n^2,1)}\left|u_n^{(2)}(x)\right|^2dx
        \leq \int_{B_R}\left|u_n^{(1)}(x)-v_1(x-z_n^1)\right|^2dx\\
        \leq &2\int_{B_R}\left|u_n^{(1)}\right|^2dx+2\int_{B_R}|v_1(x-z_n^1)|^2dx\\
        =&o_n(1)+2\int_{B(-z_n^1,R)}|v_1|^2dx=o_n(1),
\end{align*}
which is a contradiction.

Define $v_n^{(2)}(x):=u_n^{(2)}(x+z_n^2)$, combining \eqref{v2}, then there exists $v_2\in H_A^1(\R^3,\mathbb{C})\setminus \{0\}$ such that, up to a subsequence,
\begin{eqnarray}\label{4.9.0}
        \left\{\begin{array}{ll}
                v_n^{(2)}\rightharpoonup v_2 \ \ \text{weakly in}\  H_A^1(\mathbb{R}^3,\mathbb{C}),\\
                v_n^{(2)}\rightarrow v_2 \ \ \text{strongly in}\  L^p_{loc}(\mathbb{R} ^3,\mathbb{C}), \ p\in [2,6).\\
        \end{array} \right.	
\end{eqnarray}
Then, similar to \eqref{4.11}, we have $\langle I'_{\infty}(u_n^{(1)}),\varphi(x-z_n^2)\rangle=o_n(1)$ for any $\varphi\in H_A^1(\R^3,\mathbb{C})$. So, it follows from Lemma \ref{lm2.3.0} that 
\begin{align*}
   \langle  I'_{\infty}(v_n^{(2)}),\varphi \rangle 
   =&\langle I'_{\infty}(u_n^{(2)}),\varphi(x-z_n^2)\rangle 
   =\langle I'_{\infty}(v_n^{(1)}-v_1),\varphi(x+z_n^1-z_n^2)\rangle \\
    =&\langle I'_{\infty}(v_n^{(1)}),\varphi(x+z_n^1-z_n^2)\rangle- \langle I'_{\infty}(v_1),\varphi(x+z_n^1-z_n^2)\rangle+o_n(1)\\
   =&\langle I'_{\infty}(u_n^{(1)}),\varphi(x-z_n^2)\rangle+o_n(1)
   =o_n(1).
\end{align*}
According to \eqref{4.9.0}, we derive $I'_{\infty}(v_2)=0$ and hence $v_2$ is a nontrivial solution of equation \eqref{1.8.2}.
%\vskip 0.08truein

Iterating the above procedure, we can construct $\{z_n^j\}$ and $\{v_j\}$. For every $j$, $I_{\infty}(v_j)\geq c_{\infty}>0$, so the iteration must terminate at a finite number of times.
\qed
%\begin{rem}
%   By Lemma \ref{A}, we know $u_n-u_0-\sum_{j=1}^kv_j(\cdot -z_n^j)\to 0 $ in $H^1(\mathbb{R}^3,\mathbb{C})$.
%\end{rem}

According to Lemma \ref{lm4.2}, we obtain the following result.
\begin{lem}\label{lm3.6}
         Suppose $(A)$ and $(V)$ hold. If $c<c_{\infty}$, then equation \eqref{eqs1.1.0} has a ground state solution.
\end{lem}
\noindent{\bf Proof.} 
 Let \(\{w_n\}\) be a minimizing sequence of $c$. By Lemma \ref{lm3.1} and the Ekeland Variational Principle \cite[Theorem 8.5]{W}, we know that there exist $\{\lambda_n\}\subset\R $ and $\{u_n\}\subset\mathcal{N}$ such that 
 \begin{align}\label{1}
     I(u_n)\to c,\ \|I'(u_n)-\lambda_n J'(u_n)\|_*\to 0 \ \text{and}\ \|u_n-w_n\|_A\to 0 \ \text{as}\ n\to +\infty.
 \end{align}
 %Recall that $\{u_n\}\subset\mathcal{N}$, combining with \eqref{1}, 
% Similar to \eqref{b}, we obtain 
 % Moreover,
 Note that
 \begin{align*}
     %\label{eq1}
     c+o_n(1)=I(u_n)=I(u_n)-\frac{1}{4}J(u_n)=\frac{1}{4}\|u_n\|_A^2,
 \end{align*}
 which implies that $\{u_n\}$ is bounded in $H_A^1(\mathbb{R}^3,\mathbb{C})$. 
Then,
\begin{align*}
    o_n(1)=\langle I'(u_n),u_n\rangle-\lambda_n\langle J'(u_n),u_n\rangle
    =-\lambda_n\langle J'(u_n),u_n\rangle
    =2\lambda_n\|u_n\|_A^2=8\lambda_n(c+o_n(1)).
\end{align*}
%which yields 
Therefore, $\lambda_n=o_n(1)$.
Moreover, 
\begin{align*}
  \|J'(u_n)\|_*&=\sup_{\|\varphi\|_A=1}\langle J'(u_n),\varphi\rangle= \sup_{\|\varphi\|_A=1}\mathrm{Re}\left(2\int_{\R^3} [\nabla_A u_n \overline{\nabla_A \varphi} +V(x) u_n\overline{\varphi}]dx-4\int_{\R^3}\Phi_{|u_n|}u_n\overline{\varphi} dx\right)\\
  &\le C \left(\|u_n\|_A+\|u_n\|_A^3\right)\le C,  
\end{align*}
 % one has
 % \begin{align}
 %     \langle J'(u_n),u_n\rangle=2\|u_n\|_A^2-4\int_{\R^3}\Phi_{u_n}u_n^2dx=-2\|u_n\|_A^2,
 % \end{align}
 which, together with \eqref{1}, gives
 \begin{align*}
   \|I'(u_n)\|_*\leq \|I'(u_n)-\lambda_n J'(u_n)\|_*
   +|\lambda_n|\cdot \|J'(u_n)\|_*
   =o_n(1).
 \end{align*}
 Hence, \(\{u_n\}\) is a bounded \((PS)_c\) sequence of $I$. %Similar to \eqref{b}, \(\{v_n\}\) is bounded in \(H^1_A(\mathbb{R}^3, \mathbb{C})\). 
 %Furthermore, there exists \(u^0 \in H^1_A(\mathbb{R}^3, \mathbb{C})\) such that
%\[
%u_n \rightharpoonup u^0 \quad \text{weakly in } H^1_A(\mathbb{R}^N, \mathbb{C}).
%\]
%Suppose that \(u_n \not\to u^0\) in \(H^1_A(\mathbb{R}^3, \mathbb{C})\). 

From Lemma \ref{lm4.2}, there exists a solution $u^0$ of equation \eqref{eqs1.1.0} and \(k\) nontrivial solutions \(v_1, v_2, \dots, v_k \in H^1(\mathbb{R}^3, \mathbb{C})\) of equation \eqref{1.8.2} such that
\[
c = I(u^0) + \sum_{j=1}^k I_{\infty}(v_j).
\]
If $u^0\neq 0$, then $I(u^0)\geq c>0$; if $u^0=0$, then $I(u^0)=0$. So, if $k\geq 1$, then 
\[
c= I(u^0) + \sum_{j=1}^k I_{\infty}(v_j)\ge kc_{\infty}\geq c_{\infty},
\]
which is a contradiction with $c<c_{\infty}$. Hence, $k=0$ and $I(u^0)=c$, which means that $u^0$ is a ground state solution of equation \eqref{eqs1.1.0}. 
\qed 

%Conversely, when $c=c_{\infty}$, the result is different.
%Moreover, if $c$ is not attained, we have the following results.
By Lemmas \ref{lm3.3} and \ref{lm3.6}, we have the following corollary.

\begin{cor}\label{co1}
         Suppose $(A)$ and $(V)$ hold. If equation \eqref{eqs1.1.0} has no ground state solution, then $c=c_{\infty}$.
\end{cor}

Next, when equation \eqref{eqs1.1.0} has no ground state solution, we are going to prove that equation \eqref{eqs1.1.0} has a high energy solution. By Lemma \ref{lm4.2} and Theorem \ref{th1.0}, we derive the following compactness result.
\begin{lem}\label{lm4.3}
       Suppose that $(A)$, $(V)$ hold and equation \eqref{eqs1.1.0} has no ground state solution. If $I$ has a $(PS)_{d}$ sequence with $d\in (c_{\infty},c_{\infty}+\min\{c_{\infty},\tau\})$, then equation \eqref{eqs1.1.0} has a high energy solution $u_0$ with  $I(u_0)=d$.
  \end{lem}
  \noindent{\bf Proof.}
We deduce from Lemma \ref{lm4.2} that there exist a solution $u_0$ of equation \eqref{eqs1.1.0} and \(k\) nontrivial solutions \(v_1, v_2, \dots, v_k \in H^1(\mathbb{R}^3, \mathbb{C})\) of equation \eqref{1.8.2} such that
\[
d = I(u_0) + \sum_{j=1}^k I_{\infty}(v_j).
\]
Due to $d\in (c_{\infty},c_{\infty}+\min\{c_{\infty},\tau\})$, we obtain $k=0$ or $k=1$. If $k=1$ and $u_0\neq 0$, by Corollary \ref{co1}, we get
\begin{align*}
        c_{\infty}+\min\{c_{\infty},\tau\}>d=I(u_0)+I_0(v_1)\geq c+c_{\infty}=2c_{\infty},
\end{align*}
which is a contradiction.
If $k=1$ and $u_0\equiv 0$, then
\begin{align*}
    I_{\infty}(v_1)=d< c_{\infty}+\tau.
\end{align*}
According to Theorem \ref{th1.0}, we know that $v_1$ is a ground state solution of equation \eqref{1.8.2} and hence $I_{\infty}(v_1)=c_{\infty}$, 
% \begin{align*}
%         v_n^{(1)}(x)=u_n^{(1)}(x+z_n^1)=(u_n)_+(x+z_n^1)\geq 0
% \end{align*}
% and $v_n^{(1)}\rightharpoonup v_1$ in $H^1(\mathbb{R}^3)$, which implies $v_1\geq 0$. Since $v_1$ is a nontrivial solution of equation \eqref{1.5}, using the regularity of solution for equation \eqref{1.5} (see \cite[Proposition 4.1]{MV}), we obtain $v_1\in C^2(\mathbb{R}^3)$. From the strong maximum principle (see \cite{E} Chapter 6, Theorem 4), we have $v_1>0$. According to Theorem A, $v_1$ is a positive ground state solution of equation \eqref{1.5}, hence $I_0(v_1)=c_{\infty}$. Applying Lemma \ref{lm4.2}, one has $c_{\infty}<d=I_0(v_1)=c_{\infty}$, 
which is also a contradiction.

Thus, $k=0$, which yields that $I(u_0)=d$ and $u_0$ is a high energy solution of equation \eqref{eqs1.1.0}.
\qed

\vskip 0.08truein

From Lemma \ref{lm4.3}, to prove the existence of a high energy solution for equation \eqref{eqs1.1.0}, it suffices to construct a $(PS)_{d}$ sequence of $I$ with $d\in (c_{\infty},c_{\infty}+\min\{c_{\infty},\tau\})$.  Inspired by \cite{BL}, we define a barycenter function $m: H_A^1(\mathbb{R}^3,\mathbb{C})\setminus\{0\}\to \mathbb{R}^3$ by
 \begin{align*}
        m(u):=\frac{1}{\|u\|^2_{L^{2}(\mathbb{R} ^3)}}\int_{\mathbb{R}^3}\frac{x}{|x|}|u(x)|^2dx.
       % =&\left(\int_{\mathbb{R}^3}u^2(x)\chi(|x|)x_1 dx, \int_{\mathbb{R}^3}u^2(x)\chi(|x|)x_2 dx, \int_{\mathbb{R}^3}u^2(x)\chi(|x|)x_3 dx\right),
 \end{align*}
 %where
% \begin{eqnarray}
 %       \chi (t)=\left\{\begin{array}{ll}
  %      1,& 0\leq t\leq 1;\\
 %       \frac{1}{t}, & t>1.\nonumber\\
  %      \end{array} \right.	
%\end{eqnarray}
Obviously, $m$ is continuous with respect to $u$ from $H_A^1(\mathbb{R}^3,\mathbb{C})$ to $\mathbb{R}^3$.
Denote a subset of Nehari manifold $\mathcal{N}$ by
\begin{align*}
        \mathcal{N}_0:=\{u\in \mathcal{N}\ |\ m(u)=0\}.
\end{align*}
Note that $\mathcal{N}_0\neq \varnothing$. In fact, if $v\in H^1_A(\mathbb{R}^3,\mathbb{C})\setminus \{0\}$ is a radial function, we deduce from Lemma \ref{lm2.1} that there exists $t>0$ such that $tv\in \mathcal{N}$. Recall that $v$ is a radial function, then $m(v)=0$.
% $v^2(x)\chi(|x|)x_i$ is odd in $x_i$ for $i=1,2,3$, which implies
%\begin{align*}
%        \int_{\mathbb{R}^3}v^2(x)\chi(|x|)x_i dx=0,\ i=1,2,3,
%\end{align*}
%that is, $m(v)=0$.
Thus, $m(tv)=m(v)=0$ and hence $tv\in \mathcal{N}_0$. So, we can define
\begin{equation*}
        c_0:= \inf \big\{I(u)\ |\ u\in \mathcal{N}_0 \big\}.
\end{equation*}

We can get the following relationship between $c$ and $c_0$.
\begin{lem}\label{lm4.4}
       Suppose that $(A)$, $(V)$ hold and equation \eqref{eqs1.1.0} has no ground state solution. Then $c<c_0$.
\end{lem}
\begin{proof}
     %refer to wei yuanhong Lemma 5.4
Due to $\mathcal{N}_0\subset \mathcal{N}$, we get $c\leq c_0$. In order to prove $c<c_0$, we argue by contradiction and assume $c=c_0$. Then, $c_0=c=c_{\infty}$. Suppose $\{w_n\}\subset \mathcal{N}_0$ is a minimizing sequence of $c_0$. 
By the Ekeland Variational Principle \cite[Theorem 8.5]{W}, there exist $\{\lambda_n\}\subset\R $ and $\{u_n\}\subset\mathcal{N}$ such that 
 \begin{align}\label{3.24}
     I(u_n)\to c,\ \|I'(u_n)-\lambda_n J'(u_n)\|_*\to 0 \ \text{and}\ \|u_n-w_n\|_A\to 0 \ \text{as}\ n\to +\infty.
 \end{align}
 Following a similar argument as in Lemma \ref{lm3.6}, we derive that $\{u_n\}$ is a $(PS)_{c}$ sequence of $I$.
%So, %$\{|u_n|\}$ is also a minimizing sequence of $c$, and then 
 %same as discussion in the Lemma \ref{lm3.6},  we know that there exists $\{u_n\}\subset \mathcal{N}$ such that $\{u_n\}$ is a $(PS)_{c}$ sequence of $I$ on $H_A^1(\mathbb{R}^3,\mathbb{C})$.
 Since equation \eqref{eqs1.1.0} has no ground state solution, it follows from Lemma \ref{lm4.2} that there exists a solution $v_1$ of equation \eqref{1.8.2} and a sequence $\{y_n\}\subset\R^3$ with $|y_n|\to +\infty$ such that 
\begin{align}\label{3.16}
    c=c_{\infty}=I_{\infty}(v_1) \ \ \text{and}\ \ u_n-v_1(\cdot-y_n)\to 0\ \text{in} \ H_A^1(\mathbb{R}^3,\mathbb{C}) \text { as } n \to +\infty.
\end{align}
Combining \eqref{3.24} and \eqref{3.16}, we get  
\begin{align}\label{3.26}
    w_n-v_1(\cdot-y_n)\to 0\ \text{in} \ H_A^1(\mathbb{R}^3,\mathbb{C}) \text { as } n \to +\infty.
\end{align}
%Hence, from Theorem \ref{thm A} we obtain that $v_1$ is radial, then $m(v_1)=0$ 
% So, combining the continuty of $m$ and $m(u_n)=0$, one has 
% \begin{align*}
%     |m(v_1(x-y_n))|<\frac{1}{2}.
% \end{align*}
%By the continuty of $m$ and $m(u_n)=0$ we have $|m(v_1(x-y_n))|<\frac{1}{2}$

On the one hand, for $R>0$, by the mean value theorem, there exists $\theta\in(0,1)$ such that the following inequality holds
\begin{align*}
   &\left |m(v_1(x-y_n))-\frac{y_n}{|y_n|}\right|
    =\frac{1}{\|v_1\|_{L^2(\mathbb{R}^3)}^2} \left|\int_{\R^3}\left(\frac{x+y_n}{|x+y_n|}-\frac{y_n}{|y_n|}\right)|v_1|^2dx\right|\nonumber\\
    =&\frac{1}{\|v_1\|_{L^2(\mathbb{R}^3)}^2} \left|\int_{\R^3\setminus{B_R}}\left(\frac{x+y_n}{|x+y_n|}-\frac{y_n}{|y_n|}\right)|v_1|^2dx+\int_{B_R}\left(\frac{x+y_n}{|x+y_n|}-\frac{y_n}{|y_n|}\right)|v_1|^2dx\right|\nonumber\\
    \leq &C\int_{\R^3\setminus{B_R}}|v_1|^2dx+C\int_{B_R}\frac{|x|}{|\theta x+y_n|}|v_1|^2dx\nonumber\\
    \leq &C\int_{\R^3\setminus{B_R}}|v_1|^2dx+C\frac{R}{|y_n|-R}\int_{\R^3}|v_1|^2dx,
\end{align*}
which yields 
\begin{align}\label{3.17}
    \left |m(v_1(x-y_n))-\frac{y_n}{|y_n|}\right|=o_n(1).
\end{align}
On the other hand, we deduce from \eqref{3.26} that 
\begin{align*}
    &|m(v_1(x-y_n))|=|m(w_n)-m(v_1(x-y_n))|\\
    =&\left|\frac{1}{\|w_n\|^2_{L^{2}(\mathbb{R} ^3)}}\int_{\mathbb{R}^3}\frac{x}{|x|}|w_n(x)|^2dx -\frac{1}{\|v_1\|^2_{L^{2}(\mathbb{R} ^3)}}\int_{\mathbb{R}^3}\frac{x}{|x|}|v_1(x-y_n)|^2dx\right|\\
    \leq &\left|\left(\frac{1}{\|w_n\|^2_{L^{2}(\mathbb{R} ^3)}}-\frac{1}{\|v_1\|^2_{L^{2}(\mathbb{R} ^3)}}\right)\int_{\mathbb{R}^3}\frac{x}{|x|}|w_n(x)|^2dx\right|
    +\frac{1}{\|v_1\|^2_{L^{2}(\mathbb{R} ^3)}}\left|\int_{\mathbb{R}^3}\frac{x}{|x|}(|w_n(x)|^2-|v_1(x-y_n)|^2)dx\right|\\
    \leq &\left|1-\frac{\|w_n\|^2_{L^{2}(\mathbb{R} ^3)}}{\|v_1\|^2_{L^{2}(\mathbb{R} ^3)}}\right| +C\int_{\mathbb{R}^3}\left||w_n(x)|^2-|v_1(x-y_n)|^2\right|dx
    =o_n(1),
\end{align*}
which contradicts \eqref{3.17}. Thus, $c<c_0$.
\end{proof}

Next, in order to construct a min-max value belonging to $(c_{\infty},c_{\infty}+\min\{c_{\infty},\tau\})$, we define a map $f: \mathbb{R}^3\to \mathcal{N}$ by
\begin{align*}
        f(z):=\frac{\left[\int_{\R^3}\left[|\nabla u_{\infty}|^2+\left(|A(x+z)|^2+V(x+z)\right)u_{\infty}^2\right]dx\right]^{\frac{1}{2}}}{\|u_{\infty}\|}u_{\infty}(x-z),
\end{align*}
where $u_{\infty}$ is given by \eqref{1.8.3}. A direct calculation yields that $f\in C(\mathbb{R}^3, \mathcal{N})$.

We have the following properties related to $f$.
\begin{lem}\label{lm4.5}
        Suppose that $(A)$, $(V)$ hold and equation \eqref{eqs1.1.0} has no ground state solution. Then there exists $R_0>0$ such that $I(f(z))<\frac{c+c_0}{2}$ for $|z|\geq R_0$ and $\deg(m\circ f,B_{R_0},0)=1$.
\end{lem}
\noindent{\bf Proof.}
Similar to \eqref{3.3}, one has
\begin{align*}
        \lim_{|z|\to +\infty} I(f(z))=\frac{\|u_{\infty}\|^2}{4}=c_{\infty}=c.
\end{align*}
Therefore, there exists $R_1>0$ such that $I(f(z))<\frac{c+c_0}{2}$ for $|z|\geq R_1$.

Similar to \eqref{3.17}, we have%refer to Li yanyan,
\begin{align*}
\left|m(f(z))-\frac{z}{|z|}\right|
=\left|m(u_{\infty}(x-z))-\frac{z}{|z|}\right|\to 0\ \ \text{as}\ \ |z|\to +\infty.
    %\lim_{|z|\to +\infty}m(f(z))=\frac{z}{|z|}.
\end{align*}
So, there exists $R_0>\max\{1,R_1\}$ such that for $|z|\ge R_0$, one has $|m(f(z))-\frac{z}{|z|}|<\frac{1}{2}$.
Define a continuous mapping $H:[0,1]\times \overline{B_{R_0}}\rightarrow\R^3$ by $H(t,x):=tm(f(x))+(1-t)x$, then $H(t,x)\in B(\frac{tx}{|x|}+(1-t)x,\frac{t}{2})$. For $x\in \partial B_{R_0}$ and $t\in [0,1]$, we have 
% \begin{align*}
 %    &\left|H(t,x)-\frac{tx}{|x|}-(1-t)x\right|\le \frac{1}{2},
% \end{align*}
 \begin{align*}
     |H(t,x)|\ge 
    \left|\frac{tx}{|x|}+(1-t)x\right|-\frac{t}{2}
     =\frac{t}{2}+(1-t)R_0
     \neq 0,
 \end{align*}
  that is, $0\notin H(t,\partial B_{R_0})$ for $t\in [0,1]$.
 % \begin{align*}
 %     &|H(t,x)|=|tm(f(x))+(1-t)x|=\left|t\left[m(f(x))-\frac{x}{|x|}\right]+\frac{x}{|x|}\right|\\
 %     \ge& \left|1-\left|t\left[m(f(x))-\frac{x}{|x|}\right]\right|\right|\ge \frac{1}{2},
 % \end{align*}
 So, by the Brouwer degree theory, we obtain 
 \[
 \deg(m\circ f,B_{R_0}, 0)=\deg(H(1,\cdot),B_{R_0}, 0)=\deg(H(0,\cdot),B_{R_0}, 0)= \deg(Id,B_{R_0}, 0)=1,
 \]
where $Id$ is the identity mapping.
% Analogous to \eqref{4.20}, we get
% \begin{align*}
%         (m(f(z)),z)=&\frac{\|\bar{u}_{\infty}\|^2-\int_{\R^3}\Phi_0 \bar{u}_{\infty}^2(x-z) dx}{\|\bar{u}_{\infty}\|^2}\int _{\mathbb{R} ^3}\bar{u}_{\infty}^2(x-z)\chi (|x|)(x,z) dx\\
%         \geq &C\left((|z|^2-r^2)^{\frac{1}{2}}-\frac{1}{|z|}\right).
% \end{align*}
% Consequently, there exists $R_0\geq R_1$ such that $(m(f(z)),z)>0$ for $|z|=R_0$, which together with the properties of Brouwer degree, implies $\deg(m\circ f,B_{R_0},0)=1$.
\qed

Now, we are ready to construct a min-max value belonging to $(c_{\infty},c_{\infty}+\min\{c_{\infty},\tau\})$. 
Define
\begin{align}\label{3.29}
        \bar{c} :=\inf_{\gamma \in\Gamma} \max_{z\in \overline{B_{R_0}}}I(\gamma(z)),
\end{align}
where $\Gamma :=\{\gamma \in C(\overline{B_{R_0}}, \mathcal{N})\ | \ \gamma|_{\partial B_{R_0}}=f|_{\partial B_{R_0}}\}$.
%\begin{align*}
%     \Gamma :=\{\gamma \in C(\overline{B_{R_0}}, \mathcal{N})\ | \ \gamma|_{\partial B_{R_0}}=f|_{\partial B_{R_0}}\}.
%\end{align*}

\begin{lem}\label{lm4.6}
       Suppose that $(A)$, $(V)$, $(H_1)$ hold and equation \eqref{eqs1.1.0} has no ground state solution. Then $\bar{c}\in (c_{\infty},c_{\infty}+\min\{c_{\infty},\tau\})$.
\end{lem}
\noindent{\bf Proof.}
For any $\gamma \in\Gamma$, since $\gamma|_{\partial B_{R_0}}=f|_{\partial B_{R_0}}$, we deduce from the properties of Brouwer degree and Lemma \ref{lm4.5} that
\begin{align*}
        \deg(m\circ \gamma,B_{R_0},0)=\deg(m\circ f,B_{R_0},0)=1.
\end{align*}
So, there exists $z_0\in B_{R_0}$ such that $m(\gamma(z_0))=0$. 
Then
\begin{align*}
    c_0\le I(\gamma(z_0))\le \max_{z\in \overline{B_{R_0}}}I(\gamma(z)),
\end{align*}
which implies 
\begin{align}
    \label{c}
    c_{\infty}=c<c_0\le \bar{c}.
\end{align}
Moreover, since $(H_1)$ holds, we have
\begin{align*}
    I(f(z))=&\frac{1}{4}\frac{\left[\int_{\R^3}\left[|\nabla u_{\infty}|^2+\left(|A(x+z)|^2+V(x+z)\right)u_{\infty}^2\right]dx\right]^2}{\|u_{\infty}\|^2}\\
    < &\frac{1}{4\|u_{\infty}\|^{2}}\left[\int_{\R^3}\left[|\nabla u_{\infty}|^2+\left[V_{\infty}+4\|u_{\infty}\|_{L^2(\R^3)}^{-2}\left(\sqrt{c_{\infty}(c_{\infty}+\min\{c_{\infty},\tau\})}-c_{\infty}\right)\right]u_{\infty}^2\right]dx\right]^2\\
    %< &\frac{1}{4}\left[\int_{\R^3}\left[|\nabla u_{\infty}|^2+\left(\frac{c_{\infty}+\min\{c_{\infty},\tau\}}{c_{\infty}}\right)^{\frac{1}{2}}V_{\infty}u_{\infty}^2\right]dx\right]^2\|u_{\infty}\|^{-2}\\
    =&\frac{1}{4\|u_{\infty}\|^{2}}\left[\|u_{\infty}\|^{2}+4\left(\sqrt{c_{\infty}(c_{\infty}+\min\{c_{\infty},\tau\})}-c_{\infty}\right)\right]^2\\
    =&\frac{1}{16c_{\infty}}\left(4c_{\infty}+4\left(\sqrt{c_{\infty}(c_{\infty}+\min\{c_{\infty},\tau\})}-c_{\infty}\right)\right)^2\\
    =&c_{\infty}+\min\{c_{\infty},\tau\}.
    %= &\frac{1}{4}\frac{c_{\infty}+\min\{c_{\infty},\tau\}}{c_{\infty}}\|u_{\infty}\|^2=c_{\infty}+\min\{c_{\infty},\tau\}.
\end{align*}
Thus, $\bar{c}< c_{\infty}+\min\{c_{\infty},\tau\}$. The proof is completed.
\qed

\begin{lem}\label{lm3.12}
       Suppose that $(A)$, $(V)$, $(H_1)$ hold and equation \eqref{eqs1.1.0} has no ground state solution. Then equation \eqref{eqs1.1.0} has a high energy solution $u$ with $I(u)=\bar{c}$.
\end{lem}
\noindent{\bf Proof.} %refer to [Shan, CVPDE] proof of Theorem 1.3
According to Lemmas \ref{lm4.3} and \ref{lm4.6}, to show the existence of a high energy solution for equation \eqref{eqs1.1.0}, it suffices to prove that $I$ has a $(PS)_{\bar{c}}$ sequence. Using \eqref{c}, Lemmas \ref{lm4.4} and \ref{lm4.5}, we get
\begin{align*}
        \max_{z\in \partial B_{R_0}}I(f(z))<\frac{c+c_0}{2}<c_0\leq \bar{c}.
\end{align*}
Due to the general  minimax principle \cite{W}, there exists $\{w_n\}\subset \mathcal{N}$ such that
\begin{align}
\label{4.25}
        I(w_n)\to \bar{c}, \ \|I'(w_n)\| _{**}\to 0\ \text{as}\ n\to +\infty,
\end{align}
 where
 \begin{align*}
         \|I'(w)\|_{**}:=&\sup_{\psi \in T_{w}\mathcal{N} \atop \|\psi\|_A=1} \left\langle I'(w),\psi \right\rangle ,\\
         T_{w}\mathcal{N}:=&\{\psi \in H_A^1(\R^3,\mathbb{C})\ | \ \left\langle J'(w),\psi \right\rangle=0\}.
 \end{align*}
 In light of \eqref{4.25} and \cite[Proposition 5.12]{W}, one has
 \begin{align*}%\label{4.27}
         o_n(1)=\|I'(w_n)\| _{**}=\min_{\lambda \in \R}\|I'(w_n)-\lambda J'(w_n)\|_*
         =\|I'(w_n)-\lambda_n J'(w_n)\|_*.
 \end{align*}
 Following a similar argument as in Lemma \ref{lm3.6}, we obtain that $\{w_n\}$ is a $(PS)_{\bar{c}}$ sequence of $I$. Then, by Lemma \ref{lm4.3}, equation \eqref{eqs1.1.0} admits a high energy solution $u$ with $I(u)=\bar{c}$.
 This completes the proof.
\qed 

\noindent\textbf{Proof of Theorem \ref{th1.4}.}
According to Lemmas \ref{lm3.6} and \ref{lm3.12}, we get Theorem \ref{th1.4}.
\qed

\section{Proof of Theorems \ref{th1.1}-\ref{th1.2}}\label{s4}
In this section, we prove Theorems \ref{th1.1}-\ref{th1.2}.
First, we give the proof of Theorem \ref{th1.1}.

\noindent\textbf{Proof of Theorem \ref{th1.1}.}
Suppose $(A)$, $(V)$ and $(H_2)$ hold. From Lemma \ref{lm3.6}, it suffices to prove $c<c_{\infty}$. 
%refer to [Shan, CVPDE] Lemma 2.2
% To begin with, we claim there exists $\delta >0$ such that
% \begin{align}%\label{2.0}
%         J(u)>0\ \ \text{for}\ \ 0<\|u\|\leq \delta.
% \end{align}
% According to the Hardy-Littlewood-Sobolev inequality and Sobolev inequality, we get
% \begin{align*}
%         \int _{\mathbb{R}^3}\Phi_u u^2 dx\leq C\|u\|_{L^{\frac{12}{5}}(\mathbb{R} ^3)}^4\leq C\|u\|^4,
% \end{align*}
% which together with \eqref{2.1}, implies
% \begin{align*}
%         J(u)=\|u\|^2-\int_{\R^3}\left[\Phi_u u^2+\Phi_0 u^2 \right]dx
%         \geq \frac{1}{2}\|u\|^2-C\|u\|^4.
% \end{align*}
% Thus, the claim holds.
% We deduce from \eqref{2.0} that
% \begin{align}\label{2.2.0}
%         \|u\|>\delta \ \ \text{for each}\ \ u\in \mathcal{N},
% \end{align}
% which together with \eqref{2.1}, gives
% \begin{align}\label{2.2}
%         I(u)=I(u)-\frac{1}{4}J(u)=\frac{1}{4}\left(\|u\|^2-\int _{\mathbb{R}^3}\Phi_0 u^2 dx\right)
%         \geq \frac{1}{8}\|u\|^2>\frac{1}{8}\delta^2, \ u\in \mathcal{N}.
% \end{align}
% Therefore, $c\geq \frac{1}{8}\delta^2>0$.
%It follows from $(H_2)$ that there exists $x_0\in \R^3$ such that $|A(x_0)|^2+V(x_0)-V_{\infty}<0$. Then, by $A$ and $V$ are continuous, there exists $r>0$ such that for all $x\in B(x_0,r)$, we have $|A(x)|^2+V(x)-V_{\infty}<0$. 
Recall that $u_{\infty}$ is a positive ground state solution of equation \eqref{1.8.2}, so we have $\int_{\R^3}(|A(x)|^2+V(x)-V_{\infty})u_{\infty}^2dx<0$. Then
\begin{align*}
        J(u_{\infty})=J_{\infty}(u_{\infty})+\int_{\R^3}\left(|A(x)|^2+V(x)-V_{\infty}\right)u_{\infty}^2dx<J_{\infty}(u_{\infty})=0.
\end{align*}
According to Lemma \ref{lm2.1}, there exists $t\in (0,1)$ such that $tu_{\infty}\in \mathcal{N}$. Consequently,
\begin{align*}
        c\leq I(tu_{\infty})=\frac{t^4}{4}\int _{\mathbb{R}^3}\Phi_{|u_{\infty}|} u_{\infty}^2 dx
        <\frac{1}{4}\int _{\mathbb{R}^3}\Phi_{|u_{\infty}|} u_{\infty}^2 dx=I_{\infty}(u_{\infty})=c_{\infty}.
\end{align*}
We complete the proof.
\qed

\noindent\textbf{Proof of Theorem \ref{th1.2}.}
Assume $(A)$, $(V)$ and $(H_3)$ hold. Suppose by contradiction that $u\in \mathcal{N}$ is a minimizer of $c$, namely, $I(u)=c$ and $J(u)=0$.
% Moreover, by Lemma \ref{lem:1}, we know there exist $\theta_0\in \R$ and $x_0\in \R^3$ such that $u(x)=e^{i\theta_0}u_{\infty}(x-x_0)$. So, 
% \begin{align*}
%     &\|\nabla_A u\|_{L^2(\R^3)}^2=\int_{\R^3}|-i\nabla u+Au|^2dx
%     =\int_{\R^3}(-i\nabla u+Au)(i\overline{\nabla u}+A\overline{u})dx\\
%     =&\int_{\R^3}\left(|\nabla u|^2+|Au|^2\right)dx+i\int_{\R^3}(A u\cdot \overline{\nabla u}-\nabla uA\overline{u})dx
%     =\|\nabla u\|_{L^2(\R^3)}^2+\|Au\|_{L^2(\R^3)}^2.
% \end{align*} 
%First, we claim $c=c_{\infty}$. %By Lemma \ref{lm3.3}, we know $c\le c_{\infty}$. Then we are going to prove $c\ge c_{\infty}$. 
%Indeed, for any $u\in \mathcal{N}$, it holds
% Moreover, recall \eqref{nor} that 
% \[
% |\nabla_A u|^2\ge |\nabla |u||^2, \ \text{ for all } u\in H_A^1(\R^3,\C).
% \]
Thus, combining $(H_3)$ and \eqref{nor}, we have
\begin{align*}
        J_{\infty}(|u|)=J(u)+\int_{\R^3}(|\nabla |u||^2-|\nabla_A u|^2)dx +\int_{\R^3}(V_{\infty}-V(x))|u|^2dx<J(u)=0.
\end{align*}
It follows from Lemma \ref{lm2.0.1} that there exists $t\in (0,1)$ such that $t|u|\in \mathcal{N}_{\infty}$, which yields
\begin{align*}%\label{3.01}
        c_{\infty}\leq I_{\infty}(t|u|)=\frac{t^4}{4}\int_{\R^3}\Phi_{|u|} |u|^2dx<\frac{1}{4}\int_{\R^3}\Phi_{|u|} |u|^2dx=I(u)=c.
\end{align*}
This is a contradiction with Lemma \ref{lm3.3}. 
Consequently, equation \eqref{eqs1.1.0} has no ground state solution. %$c_{\infty}\leq \inf \big\{I(u)\ |\ u\in \mathcal{N} \big\}=c$, which, together with Lemma \ref{lm3.3}, yields $c=c_{\infty}$.
%By \eqref{3.01}, we have $c_{\infty}<I(u)=c$, which contradicts $c=c_{\infty}$.
\qed

%\textbf{The proof of Theorem \ref{th1.3}.} Assume $(A)$, $(V)$, $(H_1)$  and $(H_3)$ hold. Combining Theorem \ref{th1.2} and Lemma \ref{lm3.12}, we obtain Theorem \ref{th1.3}.
%\qed

\smallskip

\noindent{\bf{Data availability statement}}  

No new data was created or analysed in this study.
%Data sharing is not applicable to this article as no datasets were generated or analysed during the current study.
\smallskip

\noindent{\bf{Acknowledgments}} 

This work is supported by National Natural Science Foundation of China (No. 12471107).
\smallskip

\noindent{\bf{Conflict of interest statement}}  

All authors declare that they have no conflict of interest.
%The authors declare that there is no conflict of interest.

\end{document}